%% file: Lecture_Notes__1_.tex
\documentclass[12pt,reqno]{amsart}
\usepackage{graphicx,indentfirst}
\usepackage{amsmath,amssymb,mathrsfs}
\usepackage{amsthm,amscd}
\usepackage{verbatim}
\usepackage{appendix}
\usepackage{enumitem,titletoc}
\usepackage[utf8]{inputenc}
\usepackage{imakeidx}
\makeindex[columns=2, title=Alphabetical Index]
\usepackage{fancyhdr}
\usepackage{amsfonts,color}
\usepackage[all]{xy}
\usepackage{tikz-cd}
\usepackage{syntonly}
\usepackage{fancyhdr}
\usepackage{array}
\usepackage[left=2.5cm,right=2.5cm,bottom=3cm,top=3cm]{geometry}

\usepackage{tikz}
\usepackage{extarrows}
\usepackage{hyperref}

\def\XXint#1#2#3{{\setbox0=\hbox{$#1{#2#3}{\int}$ }
		\vcenter{\hbox{$#2#3$ }}\kern-.6\wd0}}

\newtheorem{thm}{Theorem}[section]
\newtheorem{prop}[thm]{Proposition}

\newtheorem{defn}[thm]{Definition}
\newtheorem{lem}[thm]{Lemma}
\newtheorem{cor}[thm]{Corollary}
\newtheorem{conj}[thm]{Conjecture}
\theoremstyle{remark}
\newtheorem{rk}[thm]{Remark}
\newtheorem{exercise}{Exercise}

\allowdisplaybreaks
\newcommand{\del}{\partial}
\newcommand{\dbar}{\bar\partial}

\newcommand{\ddbar}{\sqrt{-1}\partial\bar\partial}

\newcommand{\Tr}{{\rm Tr}}

\makeindex

\hypersetup{
	colorlinks=true,
	citecolor=blue,
	filecolor=green,
	linkcolor=red,
	urlcolor=black
}

\numberwithin{equation}{section}

\author{Tristan C. Collins}
\email{\href{mailto:tristanc@math.toronto.edu}{tristanc@math.toronto.edu}}
\address{Department of Mathematics, University of Toronto, 40 St. George Street, Toronto, ON, Canada}

\date{\today}
\input{derivatives}

\input{benjymacros}

\title[An introduction to Conifold Transitions]{An introduction to Conifold Transitions}

\begin{document}
\maketitle
\begin{abstract}
These lecture notes introduce conifold transitions between complex threefolds with trivial canonical bundle from the differential geometric point of view, and with a particular view towards aspects of mathematical physics and string theory.  The lecture notes are aimed at beginning graduate students and non-experts, emphasizing explicit calculations and examples.  After a brief introduction in Section~\ref{sec: intro}, we recall some basic facts about Calabi-Yau manifolds in Section~\ref{sec: CY3}. Section~\ref{sec: conifold} studies the conifold as a Calabi-Yau manifold with singularities, and introduces the local model for a conifold transition. Section~\ref{sec: GlobalConifoldTrans} discusses global conifold transitions, and recalls the famous result of Friedman \cite{Fri86} concerning the existence of smoothings for nodal Calabi-Yau threefolds.  We give a differential geometric proof of the necessity part of Friedman's theorem. Section~\ref{sec: web} discusses Reid's fantasy, and the web of Calabi-Yau threefolds. Section~\ref{sec: metric} discusses metric aspects of the local conifold transition, constructing explicit asymptotically conical Calabi-Yau metrics on the small resolution and the smoothing. Section~\ref{sec: geomConifold} discusses the metric aspects of global conifold transitions, with a particular emphasis on the heterotic string.
\end{abstract}
\section{Introduction}\label{sec: intro}

Conifold transitions, discovered by Clemens \cite{Clemens} and Friedman \cite{Friedman, Fri86}, are topology changing processes consisting of a birational contraction followed by a smoothing.  It has been proposed by Reid \cite{Reid} that these transitions can be used to connect moduli spaces of Calabi-Yau threefolds with distinct Hodge numbers, and that the resulting ``web" of Calabi-Yau manifolds is connected. In the physics literature, it has been proposed by Green-H\"ubsch \cite{GH1, GH2}, Candelas-Green-H\"ubsch \cite{CGH} that conifold transitions may unify string vaccua arising from Calabi-Yau threefolds with distinct Hodge numbers.  Strominger \cite{Stromingerholes} and Greene-Morrison-Strominger \cite{GMS} showed that, for type II string theories, this process is continuous at the level of string physics and hence could lead to a ``unified string vacuum".  These ideas have inspired the study of conifold transitions from the point of view of geometric partial differential equations, particularly those related to string vacuum equations, as proposed by Yau. The purpose of these lecture notes is to give an introduction to this circle of ideas from the perspective of differential geometry and geometric analysis, accessible to a beginning graduate student.  

Before explaining what is covered in these lecture notes, let us list the topics which are {\em not} covered.  First, we shall focus entirely on the topic of conifold transitions, ignoring completely the more general subject of {\em geometric transitions}.  However, many of the questions we will ask, (and, in a few cases, answer) have analogues in the general setting of geometric transitions.  We refer the reader to the survey article of Rossi \cite{Rossi} and the references therein for an introduction to this general circle of ideas, as well as some discussion of their relevance and importance for mathematical physics.  Secondly, our intention is that these lecture notes will be accessible to beginning graduate students, and non-experts.  For this reason, we will give very few proofs and instead emphasize examples and explicit calculations.  We have not attempted to give a comprehensive review of the many recent advances concerning the mathematical study of non-K\"ahler geometry and the heterotic, or type II string; for this we refer the reader to the survey articles \cite{Rossi, PicardSurvey1, PicardSurvey2, PhongSurvey, GF-survey} and the references therein.   Finally, we have focused completely on the complex geometric side of the story.  There is also a symplectic version of conifold transitions, pioneered by Smith-Thomas-Yau \cite{STY}.  Though this side of the story is less studied from the perspective of geometric PDE, it is equally rich and interesting.  It is only for lack of space (and time) that we have omitted it.

The plan of the lecture notes, and a brief synopsis, is as follows:
\bigskip

$\bullet$ {\em Section~\ref{sec: CY3}}:  In this section we give the definition of a Calabi-Yau threefold, and recall Yau's theorem on the existence of Ricci-flat K\"ahler metrics, as well as the Bogomolov-Tian-Todorov theorem on the structure of the moduli space of K\"ahler Calabi-Yau manifolds.  We compute the Hodge diamond of a projective hypersurface, and study the moduli space of quintic threefolds in $\mathbb{P}^4$.  
\smallskip

$\bullet$ {\em Section~\ref{sec: conifold}}: In this section we study the conifold as a Calabi-Yau manifold with singularities, and identify the local model for a conifold transition.  We exhibit two methods for ``resolving" the conifold singularity; first by small resolution and second by smoothing.  We introduce the notion of a special Lagrangian and identify the vanishing cycle of the smoothing as a special Lagrangian.  
\smallskip

$\bullet$ {\em Section~\ref{sec: GlobalConifoldTrans}}: In this section we discuss global conifold transitions, and recall the result of Friedman \cite{Fri86} concerning the existence of smoothings for nodal Calabi-Yau threefolds.  We give a differential geometric proof of (part of) Friedman's theorem, using the special Lagrangian vanishing cycles identified in Section~\ref{sec: conifold}.  We give several explicit examples of conifold transitions, and construct examples of rigid, and non-K\"ahler Calabi-Yau threefolds.
\smallskip

$\bullet$ {\em Section~\ref{sec: web}}: In this section we discuss Reid's fantasy, and the web of Calabi-Yau threefolds.  To motivate this discussion we recall some basic facts about the moduli of $K3$ surfaces.  
\smallskip

$\bullet$ {\em Section~\ref{sec: metric}}: In this section we discuss metric aspects of the local conifold transition.  We construct explicit asymptotically conical Calabi-Yau metrics on the smoothing, and the small resolution of the conifold, following constructions of Candelas-de la Ossa \cite{CdlO90}.  In particular, we observe that the local conifold transition is continuous in a metric sense.
\smallskip

$\bullet$ {\em Section~\ref{sec: geomConifold}}: In this section we discuss the metric aspects of global conifold transitions.  We introduce the heterotic string (HS) system, and show that K\"ahler Calabi-Yau manifolds with K\"ahler-Ricci flat metrics solve the HS system.  We discuss progress towards solving the HS system through a conifold transition.  We illustrate how the local geometry of the conifold transition can be used to construct solutions of (parts of) the HS system by gluing techniques, taking as a particular example the work of the author, Picard and Yau \cite{CPY}.

 \bigskip

 \noindent {\bf Acknowledgements}:  These lecture notes are based on a series of lectures given at the C.I.M.E summer school on Calabi-Yau varieties.  I am very grateful to the organizers Simone Diverio, Vincent Guedj, and Hoang Chinh Lu for the kind invitation to participate, and for their patience during the (long overdue) preparation of these notes.  Versions of these lectures were also delivered at National Taiwan University in 2024, and at the 2025 Southern California Geometric Analysis Winter School at UC Irvine.  I am grateful to Chin-Lung Wang, and Jeff Streets for their kind hospitality.  I would like to thank Sebastien Picard, Robert Friedman and Duong Phong for helpful comments on an early draft of these notes.  Finally, I am grateful to my collaborators, Sebastien Picard, Sergei Gukov, Shing-Tung Yau and Duong Phong for many fruitful discussions on conifold transitions and aspects of string theory over the past several years.  The author is supported in part by NSERC Discovery grant RGPIN-2024-518857, and NSF CAREER grant DMS-1944952.

\section{Calabi-Yau threefolds}\label{sec: CY3}

For the purposes of these lectures we will be interested mostly in complex $3$-folds.  We shall use the following strong notion of a Calabi-Yau manifold. For other more flexible notions of non-K\"ahler Calabi-Yau manifolds, see for examples \cite{Tosatti}, and the references therein.

\begin{defn}
A Calabi-Yau threefold is a simply connected complex $3$-fold $X$ with $K_{X}\sim \mathcal{O}_{X}$.  If $X$ is, in addition, K\"ahler then we shall say that $X$ is a K\"ahler Calabi-Yau threefold.
\end{defn}

We shall denote by $\Omega$ the non-vanishing holomorphic $(3,0)$ form.  Recall the following fundamental theorem, which is a special case of Yau's resolution of the Calabi conjecture.

\begin{thm}[Yau, \cite{Yau78}]\label{thm: Yau}
Let $(X,\omega)$ be a compact, K\"ahler Calabi-Yau threefold.  Then there exists a unique K\"ahler metric $\omega_{CY}$ cohomologous to $\omega$ and satisfying the complex Monge-Amp\`ere equation
\begin{equation}\label{eq: CMA}
\omega_{CY}^3 = c (\sqrt{-1})^{3^2} \Omega \wedge\overline{\Omega}.
\end{equation}
for $c \in \mathbb{R}_{>0}$. In particular, the associated Riemannian metric has zero Ricci curvature.
 \end{thm}

\subsection{Moduli spaces of K\"ahler Calabi-Yau $3$-folds}

K\"ahler Calabi-Yau manifolds typically occur in moduli spaces.  There are two obvious moduli parameters: the complex structure and the cohomology class of the K\"ahler form.  For our purposes we shall mostly be interested in the complex structure moduli. 
 Rather than attempting to give the general theory, we shall instead consider the simplest possible example; namely, the quintic threefold in $\mathbb{P}^4$.  Let $[Z_0:\cdots:Z_4] \in \mathbb{P}^4$ and consider the set
\[
X := \{P(Z_0,\ldots, Z_4)=0\} \subset \mathbb{P}^4
\]
where
\[
P(Z_0,\ldots, Z_4)= \sum_{\left\{(i_0, \ldots, i_4) \in \mathbb{Z}_{\geq 0}^5\,:\, i_0+\dots+i_4=5\right\}} a_{I} Z_0^{i_0}\cdots Z_4^{i_4}
\]
is any non-zero homogeneous polynomial of degree $5$.  For generic choices of $a_I \in \mathbb{C}$, $X$ is a smooth complex hypersurface in $\mathbb{P}^4$.  Furthermore, $K_{X} \sim \mathcal{O}_{X}$ by the adjunction formula.  We have the following lemma.

\begin{lem}\label{lem: HodgeDiamondQuintic}
The Hodge diamond of a smooth quintic hypersurface in $\mathbb{P}^4$ is 
\[
\begin{array}{ccccccc}
\,&\,&\,&1&\,&\,&\, \\
\,&\,&0&\,&0&\,&\, \\
\,&0&\,&1&\,&0&\, \\
1&\,&101&\,&101&\,&1\\
\,&0&\,&1&\,&0&\,\\
\,&\,&0&\,&0&\,&\,\\
\,&\,&\,&1&\,&\,&
\end{array}
\]
\end{lem}
\begin{proof}
This is an exercise in applying standard results from complex algebraic geometry.  In fact, we will explain a general procedure for computing the Hodge diamond of a smooth, degree $d$ hypersurface $X\subset \mathbb{P}^n$, and then specialized to the case of a quintic hypersurface only at the end.  First, by the Lefshetz hyperplane theorem and Serre duality we have that
\[
h^{n-1-p, n-1-q}(X) = h^{p,q}(X) = h^{p,q}(\mathbb{P}^4)  \quad p+q < n-1
\]
This yields all the Hodge numbers except for $h^{p,q}(X)$ for $p+q=n-1$.  We have
\[
h^{p,q}(X) = \delta_{pq} \quad p+q \ne n-1
\]
Thus, we only need to compute $h^{p,q}$ for $p+q =n-1$.   Let $\Omega_X^p$ denote the sheaf of holomorphic $p$-forms, and recall that the holomorphic Euler characteristic is given by
\begin{equation}\label{eq: holEuler}
\begin{aligned}
\chi(\Omega_{X}^p) &= \sum_{q=0}^{n-1}(-1)^q\dim H^{q}(X,\Omega_{X}^p) \\
&= \sum_{q=0}^{n-1}(-1)^qh^{p,q}(X)\\
& = (-1)^{n-1-p}h^{p,n-1-p}(X) +(-1)^p
\end{aligned}
\end{equation}
Thus, it suffices to compute $\chi(\Omega_{X}^p)$.  To do this we will use the fact that the holomorphic Euler characteristic is additive on exact sequences, together with two basic exact sequences.  The first is the Euler sequence
\begin{equation}\label{eq: EulerSEQ}
0 \rightarrow  \Omega^1_{\mathbb{P}^n} \rightarrow \mathcal{O}_{\mathbb{P}^n}(-1)^{\oplus (n+1)} \rightarrow \mathcal{O}_{\mathbb{P}^n} \rightarrow 0.
\end{equation}
Taking the wedge power of the Euler exact sequence yields (for any $0\leq p\leq n$)
\begin{equation}\label{eq: SEQ1}
0 \rightarrow  \Omega^p_{\mathbb{P}^n} \rightarrow \bigwedge^p\left(\mathcal{O}_{\mathbb{P}^n}(-1)^{\oplus (n+1)} \right)\rightarrow \Omega_{\mathbb{P}^n}^{p-1}\rightarrow 0
\end{equation}
We can expand the middle term as
\[
 \bigwedge^p\left(\mathcal{O}_{\mathbb{P}^n}(-1)^{\oplus (n+1)} \right) =  \bigwedge^p\left(\mathcal{O}_{\mathbb{P}^n}(-1)^{\oplus n} \right) \oplus \left( \mathcal{O}_{\mathbb{P}^n} (-1) \otimes \bigwedge^{p-1}(\mathcal{O}_{\mathbb{P}^n}(-1)^{\oplus n})\right).
 \]
 Applying this formula inductively shows that   
 \begin{equation}\label{eq: simpWedge}
 \bigwedge^p\left(\mathcal{O}_{\mathbb{P}^n}(-1)^{\oplus (n+1)} \right)  = \bigoplus_{i=1}^{\binom{n+1}{p}}\mathcal{O}_{\mathbb{P}^n}(-p).
 \end{equation}
Now we turn our attention to the conormal exact sequence
\begin{equation}\label{eq: conormalSEQ}
 0 \rightarrow \mathcal{O}_{X}(-d) \rightarrow \iota^*\Omega^1_{\mathbb{P}^n}\rightarrow \Omega^{1}_{X}\rightarrow 0
\end{equation} 
Taking the $p$-th wedge power yields
\begin{equation}\label{eq: conormalSEQ1}
0\rightarrow \Omega_{X}^{p-1}(-d) \rightarrow \iota^*\Omega^p_{\mathbb{P}^n} \rightarrow \Omega^{p}_{X} \rightarrow 0
\end{equation}
for $1\leq p \leq n-1$.  Finally, we have the restriction exact sequence
\begin{equation}\label{eq: restExact}
0\rightarrow \mathcal{O}_{\mathbb{P}^n}(r-d) \rightarrow  \mathcal{O}_{\mathbb{P}^n}(r) \rightarrow  \mathcal{O}_{X}(r)\rightarrow 0
\end{equation}


We can now compute $h^{p,q}(X)$ for $p+q=n-1$.
Twisting~\eqref{eq: conormalSEQ1} and taking the Euler characteristic yields
\[
\chi(\Omega_{X}^p(-r)) = \chi(\iota^*\Omega_{\mathbb{P}^n}^{p}(-r)) - \chi(\Omega_{X}^{p-1}(-r-d))
\]
On the other hand, tensoring the restriction exact sequence by $\Omega^{p}_{\mathbb{P}^n}$ we have
\[
\chi(\iota^*\Omega_{\mathbb{P}^n}^{p}(-r)) = \chi(\Omega_{\mathbb{P}^n}^{p}(-r)) - \chi(\Omega_{\mathbb{P}^n}^{p}(-r-d))  
\]
Now $\chi(\Omega_{\mathbb{P}^n}^{p}(-r)) $ can be computed inductively using~\eqref{eq: SEQ1} and~\eqref{eq: simpWedge}.  Thus $\chi(\Omega_{X}^p(-r))$ is determined by $\chi(\Omega_{X}^{p-1}(-r-d))$, and hence we can perform induction on $p$. To illustrate this we will carry out the case $p=1$, since this suffices to determined the Hodge diamond of the quintic.  From ~\eqref{eq: SEQ1} and~\eqref{eq: simpWedge} we see that, for any $r \in \mathbb{Z}_{>0}$
\[
\begin{aligned}
\chi(\Omega_{\mathbb{P}^n}(-r)) &= (n+1)\chi(\mathcal{O}_{\mathbb{P}^n}(-r-1)) - \chi(\mathcal{O}_{\mathbb{P}^n} (-r))\\
\end{aligned}
\]
Now for and $r>0$ the Kodaira vanishing theorem  and Serre duality yields
\[
\chi(\mathcal{O}_{\mathbb{P}^n} (-r))= (-1)^n\dim H^{0}(\mathbb{P}^n, \mathcal{O}_{\mathbb{P}^n}(r-(n+1))
\]
while, for $r=0$ we have
\[
 \chi(\Omega_{\mathbb{P}^n}) =-1.
 \]
Now recall that 
 \[
\dim H^{0}(\mathbb{P}^n, \mathcal{O}_{\mathbb{P}^n}(k)) = \binom{n+k}{n}.
 \]
 For simplicity let us extend the definition by $\binom{m}{n}=0$ if $m <n$.  Now we obtain
 \[
 \chi(\iota^*\Omega_{\mathbb{P}^n}) =-1 -(-1)^n\left( (n+1)\binom{d}{n}- \binom{d-1}{ n}\right)
 \]
 To compute $\chi(\mathcal{O}_{X}(-d))$ we use the restriction exact sequence~\eqref{eq: restExact} to obtain
 \[
 \chi(\mathcal{O}_{X}(-d))=\chi(\mathcal{O}_{\mathbb{P}^n}(-d))-\chi(\mathcal{O}_{\mathbb{P}^n}(-2d))
 \]
and, by Kodaira vanishing, for any $r \in \mathbb{Z}_{>0}$ we have
 \[
 \chi(\mathcal{O}_{\mathbb{P}^n}(-r))= (-1)^n\binom{r-1}{n}
 \]
 Thus, we arrive at
 \[
 \begin{aligned}
 \chi(\Omega_{X}) &= -1 -(-1)^n\left( (n+1)\binom{d}{n}- \binom{d-1}{ n}\right) + (-1)^n\left(\binom{2d-1}{n} - \binom{d-1}{n}\right)\\
 &=-1 -(-1)^n (n+1)\binom{d}{n} + (-1)^n\binom{2d-1}{n}
 \end{aligned}
 \]
 if we substitute $d=n+1$ then 
 \[
  \chi(\Omega_{X}) = -1-(-1)^n(n+1)^2+(-1)^n \binom{2n+1}{n}
  \]
  for $n=4$ this yields $  \chi(\Omega_{X}) = -1-25 +126= 100$, and so
  \[
h^{1,2} = h^{2,1} = 101
  \]
  \end{proof}
  
 Let's now count the number of parameters defining the quintic hypersurfaces in $\mathbb{P}^4$.  Naively counting the possible coefficients $a_{I}\in \mathbb{C}$ yields $126$ parameters.  However, we have over counted rather drastically.  First, note that if $P = \lambda P'$ for some $\lambda \in \mathbb{C}^*$ then $\{P=0\}= \{P'=0\}$, and so the space parameterizing quintic hypersurfaces has dimension at most $125 = 126-1$.  Next we observe that the automorphism group ${\rm Aut}(\mathbb{P}^4) = PGL(5, \mathbb{C})$ also acts on the quintic hypersurfaces, and any two hypersurfaces related by this action are isomorphic.  Since $\dim_{\mathbb{C}}=24$ we see that the space parameterizing quintic hypersurfaces has dimension at most $125 = 126-1-24=101.$   In fact, we have
 
 \begin{exercise}
 Suppose $X, X' \subset \mathbb{P}^4$ are smooth quintic hypersurfaces and there is a biholomorphic map $f: X \rightarrow X'$.  Show that there is an element $g\in PGL(5,\mathbb{C})$ such that $g\cdot X= X'$.
 \end{exercise}
 
 As a corollary of this exercise, we obtain a description of the moduli space of quintic threefolds as a Zariski open subset of $\mathbb{P}^{125}/PGL(5, \mathbb{C})$.  In particular, we have
 
 \begin{cor}
 The moduli space of smooth, quintic Calabi-Yau hypersurfaces $X\subset \mathbb{P}^4$ has dimension $101= h^{2,1}(X)$.
 \end{cor}
 
 The equality between the dimensions of the moduli space and the Hodge number $h^{2,1}(X)$ is not an accident.  The following theorem of Bogomolov-Tian-Todorov describes the local structure of the complex structure moduli space of a general K\"ahler Calabi-Yau manifold.
 
 \begin{thm}[Bogomolov \cite{Bog}, Tian \cite{Tian2}, Todorov \cite{Todorov}]\label{thm: BTT}
 Let $X$ be a smooth, K\"ahler Calabi-Yau manifold, $\dim_{\mathbb{C}}X=n$.  Then the moduli space of complex structures is locally smooth of dimension $h^{n-1,1}(X)$.
 \end{thm}
 
What is perhaps surprising is that every deformation of a quintic threefold $X\subset \mathbb{P}^4$ is achieved by a quintic threefold.  This is in stark contrast to the case of $K3$-surfaces.  Recall that a $K3$ surface is a compact, complex surface with $K_{X}\sim \mathcal{O}_{X}$ and $\pi_1(X)=\{0\}$.  For example, a smooth quartic hypersurface in $\mathbb{P}^3$ is a $K3$ surface.  The Hodge diamond of a $K3$ surface is
 \[
\begin{array}{ccccccc}
\,&\,&\,&1&\,&\,&\, \\
\,&\,&0&\,&0&\,&\,\\
\,&1&\,&20&\,&1&\,\\
\,&\,&0&\,&0&\,&\,\\
\,&\,&\,&1&\,&\,&
\end{array}
\]
 For a quartic hypersurface this can be computed using the argument in Lemma~\ref{lem: HodgeDiamondQuintic}, or, for a general $K3$ surface, by using Noether's formula.  In particular, by Theorem~\ref{thm: BTT} we see that the moduli space of $K3$ surfaces is $20$ dimensional.  On the other hand, we can easily compute the dimension of the moduli space of quartic hypersurfaces in $\mathbb{P}^3$.  One sees that there are $35$ distinct homogeneous polynomials of degree $4$ in $4$ variables.  Accounting for scaling and the action of $PLG(4,\mathbb{C})$ yields a  $35-1-15 = 19$ dimensional space parametrizing distinct quartic hypersurfaces.  In particular, we see that the space of quartic hypersurface deformations is codimension $1$ in the space of Calabi-Yau deformations. In fact, by the Torelli theorem \cite{HuybrechtsK3}, a quartic hypersurface will have deformations that are not even projective.  This observation will serve as important motivation in our consideration of Reid's fantasy in Section~\ref{sec: web}
 
It is easy to see that the moduli space of smooth quintic threefolds is not compact.  To illustrate some of the possible behaviors that can occur, consider the Dwork family
 \begin{equation}\label{eq: DworkFam}
 X_{\psi}:= \left\{\sum_{i=0}^{4}Z_i^5 - 5\psi \prod_{i=0}^{4}Z_i=0\right\} \subset \mathbb{P}^4
 \end{equation}
 where we take $\psi \in \mathbb{C}$,  but we can extend this to a family over $\mathbb{P}^1$ by setting 
 \[
 X_{\infty} =\left\{ \prod_{i=0}^{4}Z_i=0\right\} \subset \mathbb{P}^4.
 \]
The variety $X_{\infty}$ is a union of hyperplanes, and is therefore reducible and singular in complex codimension $1$.  Our interest will be in the mildly singular variety $X_1$.  The following lemma describes the singularities of $X_1$.  We leave the proof as an exercise for the reader.

\begin{lem}
Let $\xi= e^{\frac{2\pi i}{5}}$ be a primitive $5$-th root of unity.  Then for $\psi \ne \infty$ we have
\begin{itemize}
\item[(i)] If $\psi^5 \ne 1$ then $X_{\psi}$ is smooth.
\item[(ii)] The varieties $X_{\xi^k}$, $k=0,\ldots, 4$ have $125$ singular points at ${[\xi^{a_0}:\xi^{a_1}:\ldots: \xi^{a_4}]}$ for $a_i \in \mathbb{Z}_5$ and $\sum_{i=0}^4 a_i=0 \in \mathbb{Z}_5$. 
\item[(iii)] If $p\in X_1$ is a singular point, then there is a neighborhood $p\in U \subset \mathbb{P}^4$, and local holomorphic coordinates $(z_1,\ldots, z_4)$ on $U$ such that
\[
X_1 \cap U  = \{\sum_{i=1}^{4}z_i^2=0\} \cap \{\|z\| <1\} \subset \mathbb{C}^4
\]
\end{itemize}
\end{lem}

\begin{defn}
An ordinary double point is a singular point which is locally analytically isomorphic to a neighborhood of the origin in the affine variety
 \begin{equation}\label{eq: conifold}
V_0:= \left\{\sum_{i=1}^{4}z_i^2=0\right\}\subset \mathbb{C}^4.
 \end{equation}
We will also refer to such points as conifold points, or nodes.  We will refer to the affine variety in ~\eqref{eq: conifold} as the conifold.
\end{defn}

While the example of the Dwork family yields a singular quintic with $125$ nodal points, this is clearly not the generic behavior.  In fact, we have

\begin{exercise}\label{exercise: ODP}
If $X$ is a generic singular quintic, then $X$ has one ODP singularity.
\end{exercise}

\section{The Geometry of the conifold}\label{sec: conifold}

The conifold ~\eqref{eq: conifold} is a singular Calabi-Yau threefold.  We shall exhibit an explicit, non-vanishing holomorphic $(3,0)$ form on $V_0:= \left\{\sum_{i=1}^{4}z_i^2=0\right\}\subset \mathbb{C}^4$.  It is a general phenomenon that hypersurface singularities (or complete intersection singularities) admit non-vanishing holomorphic volume forms.  In the language of algebraic geometry, such singularities are said to be Gorenstein.  Explicitly, if $\{F=0\} \subset \mathbb{C}^n$ is a reduced hypersurface, the holomorphic volume form can be described as
\[
\Omega = {\rm Res}_{\{F=0\}} \frac{dz_1 \wedge \cdots \wedge dz_n}{F}
\]
Alternatively, in the set $\{\frac{\del F}{\del z_n} \ne 0\}$, define
\begin{equation}\label{eq: hypersurfaceVolForm}
\Omega= \frac{dz_1\wedge\cdots\wedge dz_{n-1}}{\frac{\del F}{\del z_n}}
\end{equation}
This formula extends in other coordinate charts (multiplying by appropriate powers of $-1$) to a global, non-vanishing holomorphic volume form.

Conifold transitions arise from the observation that ordinary double point singularity can be ``smoothed" in two topologically distinct ways.

\subsection{Smoothing the conifold by small resolution}

By a change of variables we may rewrite the conifold as the affine variety
\[
V_0:= \{xy-zw=0\} \subset \mathbb{C}^4.
\]
We blow-up along the line $\{x=z=0\}$.  Let $[U_1:U_2]$ be coordinates on $\mathbb{P}^1$, and take the closure of the graph of $\{xy=zw\}$ in $\mathbb{P}^1\times \mathbb{C}^4$ subject to the constraints $U_1z=U_2x$,
\[
\widehat{V} := \left\{ ([U_1:U_2], x,y,z,w) \in \mathbb{P}^1\times \mathbb{C}^4 : U_1z=U_2x, xy=zw \right\} \rightarrow V_0.
\]
We claim that $\widehat{V}$ is smooth.  Consider the set $\{U_2 = 1\} \subset \mathbb{P}^1$.  Over this set we can write
\[
(x,z) = z(U_1,1) \quad (w,y) = y(U_1,1).
\]
and so $\widehat{V} \cap \{U_2 =1 \} \sim \mathbb{C}^3$.  Computing similarly on $\{U_1=1\}$ shows that $\widehat{V}$ is smooth and furthermore yields a global identification $\widehat{V} = \mathcal{O}_{\mathbb{P}^1}(-1)^{\oplus 2}$, along with a map $\pi:\widehat{V} \rightarrow V_0$. 
Explicitly, this map can be given as follows.  Write 
\begin{equation}\label{eq: resolvedConifoldCoords}
\mathcal{O}_{\mathbb{P}^1}(-1)^{\oplus 2} \ni p = ([U_1:U_2], W_1, W_2)
\end{equation}
The expressions $U_iW_j.$ for $i,j =1,2$ are well-defined holomorphic functions, and hence define a map 
\begin{equation}\label{eq: resoConProj}
\begin{aligned}
\mathcal{O}_{\mathbb{P}^1}(-1)^{\oplus 2} &\rightarrow \mathbb{C}^4\\
([U_1:U_2], W_1, W_2 &\mapsto (x,y,z,w)= (U_1W_1, U_2W_2, U_1W_2,U_2W_1) \in V_0
\end{aligned}
\end{equation}
This map is an isomorphism away from $\mathbb{P}^1$ thought of as the zero section in the bundle $\mathcal{O}_{\mathbb{P}^1}(-1)^{\oplus 2}$, and the map takes $\mathbb{P}^1 \mapsto \{0\}\in \mathbb{C}^4$.  This is an example of a {\em small resolution}.  Since $\pi: \widehat{V} \rightarrow V_0$ is an isomorphism in codimension $2$, Hartog's theorem yields the following

\begin{lem}
The resolved conifold $\widehat{V} = \mathcal{O}_{\mathbb{P}^1}(-1)^{\oplus 2} $ has $K_{\widehat{V}}\sim \mathcal{O}_{\widehat{V}}$.
\end{lem}

\begin{proof}
We can write the holomorphic volume form explicitly using~\eqref{eq: resoConProj}.  For example, consider $\pi:\{U_1=1, W_2 \ne 0\} \rightarrow \{z\ne 0\} \subset V_0$ and pull-back the holomorphic volume form~\eqref{eq: hypersurfaceVolForm}
\[
\pi^*\left(\frac{dx\wedge dy\wedge dz}{z}\right) = dW_1\wedge dU_2\wedge dW_2
\]
which is clearly non-vanishing and holomorphic.  Repeating this calculation in the remaining charts on $\widehat{V}$ yields the lemma.  
\end{proof}

We end this section by noting that $\widehat{V}$ admits a rescaling action along the fibers of $\mathcal{O}_{\mathbb{P}^1}(-1)^{\oplus 2}$.  It will be convenient for us to define the rescaling map
\[
\begin{aligned}
S_{a}: \mathcal{O}_{\mathbb{P}^1}(-1)^{\oplus 2}&\rightarrow \mathcal{O}_{\mathbb{P}^1}(-1)^{\oplus 2}\\
([U_1:U_2], W_1, W_2 &\mapsto ([U_1:U_2], a^{3/2}W_1, a^{3/2}W_2)
\end{aligned}
\]

\begin{rk}
The reader will note that, in the construction of the small resolution, we made a choice to blow-up along the line $\{x=z=0\}$.  One could equally have chosen to blow-up along the line $\{x=w=0\}$. These two choices yield distinct small resolutions which are connected by a birational map.  This famous example is called the Atiyah Flop.
\end{rk}

\begin{exercise}
    Let $\widehat{V}^+$ denote the blow up of $V_0$ along the ideal $\{x=z=0\}$, and let $\widehat{V}^-$ denote the blow-up along $\{x=w=0\}$.  Show that $\widehat{V}^+$, and $\widehat{V}^-$ are birational, but not biholomorphic.
\end{exercise}

\subsection{Smoothing the conifold by deformation}

We examine a different approach to smoothing the conifold singularity.  Consider the map
\[
f:\mathbb{C}^4 \rightarrow \mathbb{C}
\]
defined by $f(z) = \sum_{i=1}^{4}z_i^2$.  This defines a family $\mathcal{V}\subset \mathbb{C}^4\times \mathbb{C} \rightarrow \mathbb{C}$ whose fiber over $t\in \mathbb{C}$ is 
\[
V_t = \{\sum_{i=1}^{4}z_i^2=t \} \subset \mathbb{C}^4.
\]
One can easily check that $V_t$ is smooth for $t\ne 0$.  The family $\mathcal{V}$ admits a rescaling action.  For $\lambda \in \mathbb{C}^*$, fix a choice of $\lambda^{1/2}$.  The particular choice will be irrelevant for our applications.  Consider the map
\[
S_{\lambda}(z) = (\lambda^{3/2}z_1, \lambda^{3/2}z_2, \ldots, \lambda^{3/2}z_4).
\]
The reason for making the admittedly odd choice of exponent $3/2$ will become apparent later when we discuss the metric geometry of the deformation family.  For now, we observe that
\begin{equation}\label{eq: rescaling}
S_{t^{1/3}}: V_1 \rightarrow V_{t}
\end{equation}
The map $S_{t^{1/3}}$ allows us to move between non-zero fibers of the smoothing family.  It turns out we can also identify $V_0$ with $V_t$ (at least away from the singular point) in a particularly convenient way. Consider the following ``nearest point projection" map
\begin{equation}\label{eq: PhitMap}
\Phi_t(z) = z + \frac{\bar{z}t}{2\|z\|^2}.
\end{equation}
Suppose $z \in V_0$.  Then we have
\[
\begin{aligned}
\Phi_t(z)\cdot \Phi_t(z) &= z\cdot z + t + t^2 \frac{\overline{z\cdot z}}{4\|z\|^4}= t
\end{aligned}
\]
and so $\Phi_t : V_0 \rightarrow V_t$.  We claim that this map defines a diffeomorphism
\[
\Phi_t(z) : V_0\cap \left\{\|z\|^2 \geq \frac{t}{2}\right\}\rightarrow V_t \setminus\{\|z\|^2=t\}
\]
We only need to show that $\Phi_t$ is injective.  We compute
\begin{equation}\label{eq: PhiMapNorm}
\begin{aligned}
\|\Phi_t(z)\|^2 &= \|z\|^2 + 2{\rm Re}\left( t\frac{z\cdot z}{2 \|z\|^2}\right) + \frac{|t|^2}{4\|z\|^2}\\
&= \|z\|^2 + \frac{|t|^2}{4\|z\|^2}
\end{aligned}
\end{equation}
The function $g(x) = x+ \frac{|t|^2}{4x}$ is strictly increasing provided $x > \frac{|t|}{2}$. Thus, if $z_1, z_2 \in V_0\cap\{ \|z\|^2 > \frac{|t|}{2}\}$ and $\Phi_t(z_1)=\Phi_t(z_2)$, then we also have $\|z_1\|=\|z_2\|$, and then ~\eqref{eq: PhitMap} implies that $z_1=z_2$.  Furthermore, one can check that
\[
\Phi_t= S_{t^{1/3}}\circ \Phi_1 \circ S_{t^{-1/3}}.
\]

The following lemma describes $V_t$ as a smooth manifold.

\begin{lem}\label{lem: defRealCoords}
For $t\ne 0$ we have $V_t \sim TS^3$.  Furthermore, for any $\epsilon \geq 0$ we have
\[
V_t \cap \{\|z\|^2=t+2\epsilon^2\} \sim S^3\times S_{\epsilon}^2
\]
where $S^2_{\epsilon} = \{ |y| = \epsilon\} \subset \mathbb{R}^3$
\end{lem}
\begin{proof}
We will construct a diffeomorphism explicitly.  For simplicity, let $t \in \mathbb{R}>0$.  The general case can be obtained from this special case by a rotation.  Write
\[
z_i = x_i + \sqrt{-1}y_i \quad i=1,\ldots, 4
\]
In terms of the real coordinates $V_t$ is given by the equations
\begin{equation}\label{eq: smoothedConifoldReal}
\sum_{i=1}^{4}x_iy_i =0 \quad \sum_{i=1}^{4} x_i^2 = \sum_{i=1}^4 y_i^2 +t.
\end{equation}
In particular, on $V_t$ for $t\ne 0$ we have $|x|^2 = \sum_{i=1}^{4} x_i^2  \geq t >0$.  Define
\[
u_i = \frac{x_i}{|x|}, \quad v_i =y_i|y|
\]
Then $u= (u_1,\ldots, u_4) \in \mathbb{R}^4$ satisfy $|u|^2=1$, while $v \in \mathbb{R}^4$ satisfies $u\cdot v=0$.  This is clearly $TS^3$. 

 Next we consider the intersection of $V_t$ with $\|z\|^2 = t+2\epsilon^2$.  By~\eqref{eq: smoothedConifoldReal} this yields the equations
\begin{equation}\label{eq: smoothedConifoldRealLink}
\vec{x}\cdot \vec{y} =0 \quad |x|^2= |y|^2 +t, \quad |y|^2=\epsilon^2.
\end{equation}
\end{proof}

The $3$-sphere $V_t \cap \{\|z\|^2=t\}$ is a vanishing cycle for the degeneration $V_t \rightarrow V_0$, as can be easily seen from the description of the degeneration in Lemma~\ref{lem: defRealCoords}.  This $3$-sphere turns out to play a critical role in understanding the smoothing of nodal Calabi-Yau $3$-folds, as we shall see later.  Recall the following definition due to Harvey-Lawson \cite{HarveyLawson}

\begin{defn}\label{defn: slag}
Suppose $(X,\omega,\Omega)$ is a K\"ahler Calabi-Yau manifold with $\dim_{\mathbb{C}}X=n$.  A real submanifold $L\subset X$ with $\dim_{\mathbb{R}}L=n$ is:
\begin{itemize}
\item[(i)] Lagrangian if $\omega\big|_{L}=0$.
\item[(ii)] Special Lagrangian (sLag) if there exists $e^{\sqrt{-1}\theta} \in \mathbb{S}^1$ such that
\[
{\rm Im}\left( e^{-\sqrt{-1}\theta}\Omega|_{L}\right) =0
\]
\end{itemize}
\end{defn}

If we assume in addition that $\omega = \omega_{CY}$ is a Calabi-Yau metric satisfying the complex Monge-Amp\`ere equation~\eqref{eq: CMA}, then special Lagrangians are a special class of {\em calibrated submanifolds}.  By the theory of calibrations developed by Harvey-Lawson \cite{HarveyLawson} such manifolds are automatically volume minimizing in their homology class. 

\begin{lem}[Harvey-Lawson \cite{HarveyLawson}]
Suppose $(X,\omega,\Omega)$ is a K\"ahler Calabi-Yau manifold, and $L\subset X$ is a compact special Lagrangian, then $L$ is volume minimizing in its homology class.  Furthermore, we have
\[
{\rm Vol}(L) = \int_{L}{\rm Re}\left(e^{-\sqrt{-1}\theta}\Omega|_{L}\right).
\]
\end{lem}

\begin{lem}\label{lem: sLagMod}
Let $L_{t} = \{ \|z\|^2 =t\} \subset V_t$ be the vanishing cycle of the degeneration $V_t\rightarrow V_0$. For $t \in \mathbb{C}^*$, write $t=|t|e^{\sqrt{-1}\theta}$.  Then
\[
{\rm Im}\left(e^{-\sqrt{-1}\theta}\Omega_t\right) =0,\quad \text{ and } \quad \int_{L_t} \Omega_t =2\pi^2t.
\]
\end{lem}
\begin{proof}
We check this formula at $t=1$.  Consider the open set $\{z_4 \ne 0\} \subset V_1$.  Working in the real coordinates introduced in Lemma~\ref{lem: defRealCoords}, $L_1 = \{|y|=0\}$, and so the holomorphic volume form satisfies
\[
\Omega_{1}|_{L_1} = \frac{dx_1\wedge dx_2 \wedge dx_3}{x_4}
\]
We use this non-vanishing form to define an orientation on $S^3$.  Then, over $S^3 \subset \mathbb{R}^4$ yields the result.  For general $t$ the result follows from the rescaling action described in~\eqref{eq: rescaling} since
\[
S_{t^{-1/3}}:L_t \rightarrow L_1, \qquad S_{t^{-1/3}}^*\Omega_1 = t\Omega_t
\]
\end{proof}

If we use the flat metric on $\mathbb{R}^8$ to identify $TS^3 \sim T^*S^3$, then $L_t\subset V_t$ is special Lagrangian in the sense of Definition~\ref{defn: slag}, however this symplectic structure is not Calabi-Yau and so $L_t$ is not minimal.  Later we will see that $V_t$ can be equipped with a Calabi-Yau structure such that $L_t$ is special Lagrangian and volume minimizing.  

\subsection{The local conifold transition}

We can now describe the local model of a conifold transition.  We consider the process
\[
\widehat{V} \rightarrow V_0 \leadsto V_t
\]
where $\widehat{V} \rightarrow V_0$ contracts $\mathbb{P}^1 \subset \widehat{V}$, followed by the deformation $V_0 \leadsto V_t$ smoothing the resulting ODP singularity.  This process allows us to pass between the topologically distinct Calabi-Yau manifolds $\widehat{V}$ and $V_{t}, t\ne 0$. 

\section{Global conifold transitions}\label{sec: GlobalConifoldTrans}

Suppose now that we have a compact complex space $X_0$ of complex dimension $3$ with only ODP singularities, and such that $K_{X_0}\sim \mathcal{O}_{X_0}$ (that is, $X_0$ is Gorenstein, with trivial canonical bundle).  From the local model it is not hard to see that one can construct a small resolution
\[
\pi:\widehat{X} \rightarrow X_0
\]
and $\widehat{X}$ is a compact, complex manifold with $K_{\widehat{X}}\sim \mathcal{O}_{\widehat{X}}$.  One can then ask whether it is possible to find a deformation family $\mathcal{X} \rightarrow \Delta = \{ t\in \mathbb{C} : |t|<1\}$ such that $X_0$ is the fiber over $0$, and the fiber $X_t, t\ne 0$ is a smooth, compact complex manifold with $K_{X_t}\sim \mathcal{O}_{X_t}$.  This is the content of a famous result of Friedman \cite{Fri86}.

\begin{thm}[Friedman \cite{Fri86}]\label{thm: Friedman}
Let $X_0$ be a compact Calabi-Yau threefold with ODP singularities.  Let $\pi:\widehat{X} \rightarrow X_0$ be a small resolution, and let $C_i$, $1 \leq i \leq k$ be the $(-1,-1)$ curves contracted by $\pi$.  Then $X_0$ admits a first-order smoothing $X_0\leadsto X_t$ if and only if there exists $\lambda_i \in \mathbb{C}^*$ for $1 \leq i \leq k$ such that
\begin{equation}\label{eq: FriedmanRelation}
\sum_{i=1}^{k}\lambda_i[C_i] =0 \in H_2(\widehat{X},\mathbb{C})
\end{equation}
\end{thm}

As explained in \cite{Kont, RolTh}, the set of classes in $H_{2}(\widehat{X},\mathbb{C})$ satisfying Friedman's relation~\eqref{eq: FriedmanRelation} should be viewed as the appropriate {\em definition} of $H^{2,1}(X_0, \mathbb{C})$.  With this perspective, it turns out that, as in the case of compact, K\"ahler Calabi-Yau manifolds, the deformation theory is unobstructed for nodal Calabi-Yau threefolds.  This was established independently by Kawamata \cite{Kawamata}, Ran \cite{Ran} and Tian \cite{Tian}:

\begin{thm}[Kawamata \cite{Kawamata}, Ran \cite{Ran}, Tian \cite{Tian}]\label{thm: unobstructed}
In the setting of Friedman's theorem, assume in addition that $\widehat{X}$ is K\"ahler, or satisfies the $\ddbar$-lemma.  Then any first order smoothing integrates to a genuine smoothing.
\end{thm}

We now give a differential geometric proof of the necessity part of Theorem~\ref{thm: Friedman}.  This proof is inspired in part by the arguments of Rollenske-Thomas \cite{RolTh}, Kontsevich \cite{Kont} and calculations of Tian \cite{Tian}.
\begin{proof}[Proof of necessity in Theorem~\ref{thm: Friedman}]
Consider the local model.  The map $\Phi_t$ introduced in~\eqref{eq: PhitMap} maps
\[
\Phi_t(z) : V_0\cap \left\{\|z\|^2 \geq \frac{t}{2}\right\}\rightarrow V_t \setminus\{\|z\|^2=t\}.
\]
Pulling back $\Omega_t$ by $\Phi_t$ and expanding in $t$ yields
\begin{equation}\label{eq: pfFriedmanVolExpan}
\Phi_t^*\Omega_t = \Omega_0 + t \widetilde{\Omega}_1 + \sum_{k \geq 2} t^k\widetilde{\Omega}_k
\end{equation}
where each $\tilde{\Omega}_k$ is smooth $3$-form on $V_0\setminus \{0\}$.  Direct calculation shows that $\widetilde{\Omega}_1$ has components of type $(3,0)$ and $(2,1)$ only.  It will be useful to have a formula for $\widetilde{\Omega}_1$.  On $\{z_4\ne 0\} \cap V_t$ the holomorphic volume form is given by
\begin{equation}\label{eq: pfFriedmanVolForm}
\Omega_t=\frac{dz_1\wedge dz_2\wedge dz_3}{z_4}
\end{equation}
Pulling back by $\Phi_t$ yields
\begin{equation}\label{eq: computeDefForm}
\begin{aligned}
\Phi_t^*\Omega_t &= \frac{1}{z_4+\frac{\overline{z_4}t}{2\|z\|^2}} \left(d(z_1+\frac{\overline{z_1}t}{2\|z\|^2}) \wedge d(z_2+\frac{\overline{z_2}t}{2\|z\|^2}) \wedge d(z_3+\frac{\overline{z_3}t}{2\|z\|^2})\right)\\
&= \Omega_0 -t\frac{\overline{z_4}}{2z_4^2\|z\|^2} dz_1\wedge dz_2 \wedge dz_3\\
&\quad+ \frac{t}{z_4} \left(d\left(\frac{\overline{z_1}}{2\|z\|^2}\right)\wedge dz_2\wedge dz_3 + dz_1\wedge d\left(\frac{\overline{z_2}}{2\|z\|^2}\right)\wedge dz_3 + dz_1\wedge dz_2 \wedge d\left(\frac{\overline{z_3}}{2\|z\|^2}\right)\right)\\
&\quad +\text{ higher order terms}
\end{aligned}
\end{equation}
and so
\[
\begin{aligned}
\widetilde{\Omega}_1 &= \frac{\overline{z_4}}{2z_4^2\|z\|^2} dz_1\wedge dz_2 \wedge dz_3\\
&\quad+ \frac{1}{z_4} \left(d\left(\frac{\overline{z_1}}{2\|z\|^2}\right)\wedge dz_2\wedge dz_3 + dz_1\wedge d\left(\frac{\overline{z_2}}{2\|z\|^2}\right)\wedge dz_3 + dz_1\wedge dz_2 \wedge d\left(\frac{\overline{z_3}}{2\|z\|^2}\right)\right)\\
\end{aligned}
\]
Let $M \subset V_0$ be any $3$-sphere such that $\Phi_t(M)$ is homologous to the vanishing cycle $L_t$ in $V_t$.  Concretely, choose a point $z_0 \in V_0 \cap \{\|z\|^2=s\}$ and consider the collection of points $z \in V_0 \cap \{\|z\|^2=s\}$ such that ${\Im}(z-z_0)=0$. Note that such $3$-spheres are precisely the fibers of an $S^3$-bundle over $S^2$, by the calculation of Lemma~\ref{lem: defRealCoords}.  Combining this observation with the construction of $\Phi_t$, in particular ~\eqref{eq: PhiMapNorm}, one can easily check that $\Phi_t(M)$ is homologous to the vanishing cycle $L_t$ for $t <s$.

By Lemma~\ref{lem: sLagMod} we have
\[
\int_{M}\Phi_t^*\Omega_t = \int_{\Phi_t(M)}\Omega_t = \int_{L_t}\Omega_t= 2\pi^2t
\]
and so we must have $\int_{M}\widetilde{\Omega}_1= 2\pi^2$, and $\int_{M}\Omega_0 =\int_{M}\widetilde{\Omega}_k=0 $ for all $k \geq 2$. 

Now we observe that in ~\eqref{eq: computeDefForm}, the form $\widetilde{\Omega}_1$ is invariant under rescaling $V_0$.  To see what this implies let
\[
\nu:\widehat{V}=\mathcal{O}_{\mathbb{P}^1}(-1)^{\oplus 2} \rightarrow V_0
\]
be the small resolution of $V_0$, and let
\[
\pi:\mathcal{O}_{\mathbb{P}^1}(-1)^{\oplus 2} \rightarrow \mathbb{P}^1
\]
be the projection, and wite $[\mathbb{P}^1]$ for the current of integration over $\mathbb{P}^1 \subset \widehat{V}$.  We claim that
\[
d(\nu^*\widetilde{\Omega}_1) = 2\pi^2 [\mathbb{P}^1].
\]
Indeed, since $d \widetilde{\Omega}_1=0$ we certainly have that
\[
d(\nu^*\widetilde{\Omega}_1) =0 \quad \text{ on }\quad  \mathcal{O}_{\mathbb{P}^1}(-1)^{\oplus 2} \setminus \mathbb{P}^1.
\]
We only need to evaluate the behaviour along $\mathbb{P}^1$.  To do this, let $\beta$ be any compactly supported smooth $2$-form on $\widehat{V}$.  Write
\[
\beta= \beta_0 + \mathcal{E}
\]
where $\beta_0 = \pi^* (\beta|_{\mathbb{P}^1})$, and $\mathcal{E} = \beta -\beta_0$ is a two-form whose restriction to $\mathbb{P}^1$ vanishes.  Note that for degree reasons, $\beta_0$ is closed.  Let
\[
N_{\epsilon} = \nu^{-1} (V_0 \cap \{ \|z\| <\epsilon\})
\]
be an $\epsilon$ neighborhood of $\mathbb{P}^1$.  Then we have
\[
\int_{N_{\epsilon}} d(\nu^*\widetilde{\Omega}_1) \wedge \beta_0 = \int_{\del N_{\epsilon}}\nu^*\widetilde{\Omega}_1 \wedge \beta_0
\]
using that $\beta_0$ is closed.  On the other hand, by Hartog's theorem the rescaling of $V_0$ lifts to a rescaling along the fibers of $\widehat{V}$, with the property that rescaling by $t$ maps $\del N_{\epsilon} \rightarrow \del N_{t\epsilon}$. Since $\widetilde{\Omega}_1$ and $\beta_0$ are both invariant under this rescaling, we conclude that
\[
\int_{\del N_{\epsilon}}\nu^*\widetilde{\Omega}_1 \wedge \beta_0=\int_{\del N_{\epsilon'}}\nu^*\widetilde{\Omega}_1 \wedge \beta_0
\]
for all $\epsilon, \epsilon'$.  Thus, we have
\[
\lim_{\epsilon \rightarrow 0}\int_{N_{\epsilon}} d(\nu^*\widetilde{\Omega}_1) \wedge \beta_0 = \lim_{\epsilon \rightarrow 0} \int_{\del N_{\epsilon}}(\nu^*\widetilde{\Omega}_1) \wedge \beta_0 = \int_{\del N_{\epsilon'}}(\nu^*\widetilde{\Omega}_1) \wedge \beta_0
\]
The latter integral can be evaluated explicitly. We note that $\del N_{\epsilon'} = \nu^{-1}( \{ \|z\|=\epsilon'\})$ is precisely the trivial $S^3$ bundle over $\mathbb{P}^1$ defined by the sections of length $\epsilon'$, measured with respect to the Fubini-Study metric on in $\mathcal{O}_{\mathbb{P}^1}(-1)^{\oplus 2}$.  Since the fibers of this fibration are homologous to $\nu^{-1}(M)$, and
\[
\int_{M}\widetilde{\Omega}_1 =2\pi^2
\]
we conclude that
\[
\int_{\del N_{\epsilon'}}(\nu^*\widetilde{\Omega}_1) \wedge \beta_0 = 2\pi^2 \int_{\mathbb{P}^1} \beta_0.
\]
Finally, we need to consider the error term.  Again we compute
\[
\int_{N_{\epsilon}} d(\nu^*\widetilde{\Omega}_1) \wedge \mathcal{E} = \int_{\del N_{\epsilon}}\nu^*\widetilde{\Omega}_1 \wedge \mathcal{E} - \int_{N_{\epsilon}} \nu^*\widetilde{\Omega}_1 \wedge d\mathcal{E}
\]
Using the $\mathbb{C}^*$ action on $\widehat{V}$ we can decompose $\mathcal{E}, d\mathcal{E}$ into a sum of homogeneous forms.  Since $\mathcal{E}$ vanishes along $\mathbb{P}^1$, each term in the sum has degree at least $1$
\[
\mathcal{E} = \sum_{k \geq 1}\mathcal{E}_k, \qquad \sum_{k\geq 1}(d\mathcal{E})_k
\]
and then, by rescaling
\[
\begin{aligned}
\int_{\del N_{\epsilon}}\nu^*\widetilde{\Omega}_1 \wedge \mathcal{E}_k &= \epsilon^k \int_{\del N_{1}}\nu^*\widetilde{\Omega}_1 \wedge \mathcal{E}_k\\
\int_{N_{\epsilon}}\nu^*\widetilde{\Omega}_1 \wedge (d\mathcal{E})_k &= \epsilon^k \int_{ N_{1}}\nu^*\widetilde{\Omega}_1 \wedge (d\mathcal{E})_k
\end{aligned}
\]
from which it follows that
\begin{equation}\label{eq: noErrorcontribution}
\lim_{\epsilon \rightarrow 0}\int_{N_{\epsilon}} d(\nu^*\widetilde{\Omega}_1) \wedge \mathcal{E}=0.
\end{equation}

We now use this local calculation to examine the setting of a global conifold transition.  Roughly speaking, the strategy is to combine the local model calculation with the existence of a global holomorphic $(3,0)$ form to recapture Friedman's condition.

Suppose that $\pi:\mathcal{X} \rightarrow \Delta=\{t\in \mathbb{C} : |t|<1\}$ is a family such that:
\begin{itemize}
    \item[(i)] the total space $\mathcal{X}$ is a smooth complex manifold, 
    \item[(ii)] for $t\in \mathbb{C}^*$ the fibers $\pi^{-1}(t)$ are compact, complex threefolds admitting global, non-vanishing holomorphic $(3,0)$ forms, $\Omega_{X_t}$
    \item[(iii)] $X_0 = \pi^{-1}(0)$ has only nodal singularities.  
    \end{itemize}
    
Let $p_i \in X_0$ be a node and let $\mathcal{U}_{p_i} \subset \mathcal{X}$ be an open neighborhood in which $\mathcal{X}$ is biholomorphic to a neighborhood of the origin in the family
\[
\mathcal{V} = \{(z,t) \in \mathbb{C}^4\times \Delta : \sum_{i=1}^4 z_i^2 -t =0\}
\]
with the projection $\pi$ given by projection to the $t$-coordinate.  Denote
\[
\mathfrak{U}(t) = \bigcup_i X_0\cap \left(\mathcal{U}_{p_i} \setminus \left\{\|z\|^2 >\frac{|t|}{2}\right\} \right)
\]
a neighborhood of the nodes in $X_0$ and let 
\[
L_i(t) = \{z \in X_t: \|z\|^2=t\}\subset X_t.
\]
denote the vanishing cycles in $X_t$.  Near each node $p_i$ we have the map $\Phi_{t}$ from our local model scenario.  Using the flow of a vector field (see, e.g. \cite[Lemma 2.13]{CPY})  we can easily extend these locally defined maps to a globally defined map
\begin{equation}\label{eq: fiberDiffeo}
F_t: X_0 \setminus\mathfrak{U}(t) \rightarrow X_t \setminus \left(\bigcup_i L_i(t)\right)
\end{equation}
Let $\nu:\widehat{X}\rightarrow X_0$ be a small resolution with exceptional rational curves $C_i$ over each node $p_i$.  One the one hand we have
\[
\frac{d}{dt}\bigg|_{t=0} (\nu^* F_{t}^*\Omega_{X_t}) = d\left(\nu^*\iota_{V}\Omega_0\right)
\]
where $V$ is the vector field whose time $t$ flow defines the map $F_t$.
Our goal is to show that the local calculation we performed above implies
\[
\frac{d}{dt}\bigg|_{t=0} (\nu^* F_{t}^*\Omega_{X_t}) = \sum_i \lambda_i[C_i]
\]
for $\lambda_i\in \mathbb{C}^*$.  Combining these two formulae give
\[
\sum_i \lambda_i[C_i]=d\left(\nu^*\iota_{V}\Omega_0\right)
\]
which yields Friedman's relation.

The only thing left to prove is that the local calculation accurately represents the global situation.  We work in the open set $\mathcal{U}_{p_i}$ near a fixed node $p_i \in X_0$.  It is not hard to show  (see, eg. \cite[Lemma 4.3]{CGPY}) that we can write
\[
\Omega_{X_t} = h(z,t)\Omega_{t}
\]
where $\Omega_t$ is the model holomorphic volume form (see e.g. ~\eqref{eq: pfFriedmanVolForm}), and  $h(z,t): \mathcal{U} \rightarrow \mathbb{C}^*$ is functions which is holomorphic in $z$ and smooth in $t$, away from the node.  Let $\tau_i = h(0,0) \in \mathbb{C}^*$.
By smoothness, the functions $h_t= \frac{\del h}{\del t}$ and $h_{\bar{t}} = \frac{\del h}{\del \bar{t}}$ are bounded uniformly in $t$ on compact sets away from the node.  On the other hand, since $h_t, h_{\bar{t}}$ are both holomorphic, this bound extends over the node by Hartog's theorem.  Now we compute locally near a node
\begin{equation}\label{eq: pfFriedmanErrorTerms}
\frac{d}{dt}\big|_{t=0}\Phi_{t}^*\Omega_{X_{t}} = \frac{\del h}{\del z}\frac{z}{2\|z\|^2} \Omega_{0} + \frac{\del h}{\del t} \Omega_0 + (h(z,0)-\tau_i))\widetilde{\Omega}_1 + \tau_i\widetilde{\Omega}_1
\end{equation}
where $\widetilde{\Omega}_1$ is defined in~\eqref{eq: pfFriedmanVolExpan}, and computed explicitly in~\eqref{eq: computeDefForm}.  Let $\nu: \widehat{X} \rightarrow X_0$ be a small resolution, and $C_i \subset \widehat{X}$ the rational curve such that $\nu(C_i)=p_i$.  We claim that, in  $\nu^{-1}(\mathcal{U}_{p_i})$ there holds
\[
d\left(\nu^*\left(\frac{d}{dt}\big|_{t=0}\Phi_{t}^*\Omega_{X_{t}}\right)\right) = 2\pi^2\tau_i[C_i].
\]
Indeed, our local calculation yields
\[
d\widetilde{\Omega}_1 = 2\pi^2[C_i]
\]
so the only thing we need to show is that the first three terms on the right hand side of~\eqref{eq: pfFriedmanErrorTerms} do not contribute.  As before, closedness implies that on $\nu^{-1}(\mathcal{U}_{p_i})\setminus C_i$ there holds
\[
d \nu^*\left(\frac{\del h}{\del z}\frac{z}{2\|z\|^2} \Omega_{0} + \frac{\del h}{\del t} \Omega_0 + (h(z,0)-h(0,0))\widetilde{\Omega}_1\right) =0
\]
and so we only need to check if this form carries any mass on $C_i$.  This follows from considerations of scaling and homogeneity, using arguments similar to those used to justify~\eqref{eq: noErrorcontribution}.  First we note that $\Omega_0$ is homogeneous of degree $2$ under rescaling, and so the uniform bounds on $\frac{\del h}{\del z}$ and $\frac{\del h}{\del t}$ imply that
\[
d\left(\frac{\del h}{\del z}\frac{z}{2\|z\|^2} \Omega_{0} + \frac{\del h}{\del t} \Omega_0\right)
\]
carries no mass on $C_i$, and hence vanishes identically.  Similarly, since $(h(z,0)-h(0,0))$ vanishes on $C_i$, the final term also carries no mass on $C_i$, and vanishes identically.

\end{proof}

\begin{rk} 
From the above argument the constants $\lambda_i \in \mathbb{C}^*$ appearing in Theorem~\ref{thm: Friedman} are precisely given by
\[
\lambda_i = \lim_{t\rightarrow 0} \frac{1}{t}\int_{L_i(t)}\Omega_{X_t}.
\]
In particular, the smoothing is determined essentially by {\em special Lagrangian} data.  This can be easily deduced from the above calculation together with Lemma~\ref{lem: sLagMod}.  This was already observed in \cite{Kont, RolTh}
\end{rk}

\begin{rk}
    The above calculation generalizes to arbitrary dimensions to yield a necessary condition for the existence of a smoothing of a Calabi-Yau $n$-fold with nodal singularities, recovering some results of Rollenske-Thomas \cite{RolTh}.  It is interesting to note that, due to scaling, one only obtains a non-trivial obstruction to smoothing when $n$ is odd. 
\end{rk}

\begin{rk}
Theorem~\ref{thm: Friedman} and Theorem~\ref{thm: unobstructed} have recently been extended to higher dimensions, under various assumptions, by Friedman-Laza, as a consequence of their study of {\em higher Du Bois} and $k$-liminal singularities; see \cite{FrLa1, FrLa2, FrLa3, FrLa4, FrLa5}. It would be interesting to provide a differential geometric interpretation of some of these results, at least in some model cases.  Further general results on the unobstructedness of the deformation theory of singular Calabi-Yau varieties, extending Theorem~\ref{thm: unobstructed}, have recently been obtained by Imagi \cite{Imagi} and Friedman \cite{Friedman2}.  
\end{rk}

The next result describes the global topology change resulting from a conifold transition, see e.g. \cite{Tian, NamSt, Rossi}

\begin{prop}\label{prop: topChange}
Let $\widehat{X} \rightarrow X_0 \leadsto X_t$ be a conifold transition where $\widehat{X} \rightarrow X_0$ contracts $N$ disjoint $(-1,-1)$ curves, and let $L_i$ $1 \leq i \leq N$ be the vanishing cycles of the smoothing $X_0 \leadsto X_t$.  Define
\begin{itemize}
\item $N= \#\{{\rm Sing}X_0\}$
\item $k= \dim_{\mathbb{R}} \{ {\rm Span}\{ [C_i]\}_{1\leq i \leq N} \} \subset H_2(\widehat{X}, \mathbb{R})$
\item $c= \dim_{\mathbb{R}} \{ {\rm Span}\{ [L_i]\}_{1\leq i \leq N} \} \subset H_3(X_t, \mathbb{R})$
\end{itemize}
Then we have
\begin{equation}\label{eq: topChange}
b_1(X_t)=b_1(\widehat{X}), \quad b_2(X_t) = b_2(\widehat{X})-k \quad b_3(X_t) = b_3(\widehat{X})+2c, \quad N= k+c
\end{equation}
Furthermore, the Hodge numbers change according to
\begin{equation}\label{eq: hodgeChange}
h^{2,1}(X_t) = h^{2,1}(\widehat{X})+c \quad h^{1,1}(X_t) = h^{1,1}(\widehat{X})-k.
\end{equation}
\end{prop}

\begin{rk}
We remark that the equations~\eqref{eq: topChange} are purely topological, and do not depend on the presence of an integrable complex structure on $X_t$.  
\end{rk}

\subsection{Examples of conifold transitions}

In this section we are going to describe several examples of conifold transitions to illustrate the many interesting phenomena which can arise.  Constructing interesting complex manifolds through conifold transitions was and idea first pioneered by Clemens \cite{Clemens} and Friedman \cite{Fri86, Friedman}.

\subsubsection{The generic nodal quintic}
Let $X_0$ denote the generic nodal quintic, which we recall has a single node, by Exercise~\ref{exercise: ODP}.  Let $\pi:\widehat{X} \rightarrow X_0$ be a small resolution and let $[C]\subset H_2(\widehat{X}), \mathbb{C})$ denote the class of the exceptional $\mathbb{P}^1$.  Since $X_0$ clearly admits a smoothing, Friedman's theorem implies that $[C]=0 \in H_2(X, \mathbb{C})$.  This immediately implies that $\widehat{X}$ is non-K\"ahler.  To see this, let $C=\del D$ and suppose we have a K\"ahler form $\omega$ on $\widehat{X}$.  Then by Wirtinger's inequality we get
\[
0<{\rm Vol}_{\omega}(C) = \int_{C}\omega = \int_{D} d\omega =0
\]
a contradiction.

\subsubsection{Schoen's rigid Calabi-Yau}
Consider the quintic 
\[
X_0:= \left\{ \sum_{i=0}^{4} z_i^5 - 5 \prod_{i=0}^{4}z_i\right\} \subset \mathbb{P}^4
\]
which has $125$ nodes.  Schoen \cite{Schoen} showed that $X_0$ admits a {\em projective} small resolution $\widehat{X}\rightarrow X_0$ having $h^{1,1}= 25$ and $h^{2,0}=h^{0,2}=0$.  The smoothing  $X_0 \leadsto X_t $ has $125$ vanishing cycles, and by Proposition~\ref{prop: topChange}, these span a $101$ dimensional space in $H_3(X_t, \mathbb{C})$.  On the other hand, this implies that $b_3(\widehat{X})= 2$.  Thus, $\widehat{X}$ has Hodge diamond
\[
\begin{array}{ccccccc}
\,&\,&\,&1&\,&\,&\, \\
\,&\,&0&\,&0&\,&\, \\
\,&0&\,&25&\,&0&\, \\
1&\,&0&\,&0&\,&1\\
\,&0&\,&25&\,&0&\,\\
\,&\,&0&\,&0&\,&\,\\
\,&\,&\,&1&\,&\,&
\end{array}
\]
This shows that $\widehat{X}$ is rigid, in the sense that it admits no nontrivial deformations at all.  Such manifolds are interesting since mirror symmetry would seem to suggest their mirrors are non-K\"ahler.

\subsubsection{The mirror quintic}\label{subsec: mirrorQuintic}

This example considers the mirror quintic, which is one of the first instances of mirror symmetry.  Consider the Dwork family~\eqref{eq: DworkFam}. For $\psi^5 \ne 1$, $X_{\psi}$ is smooth.   Let $\xi$ be a primitive $5$-th root of unity. The group
\[
G= \{(a_0,\dots, a_4) \in \mathbb{Z}_{5} : \sum_{i=0}^4 a_i=0\} /\mathbb{Z}_{5}
\]
acts on $\mathbb{P}^4$ by
\[
[Z_0:\ldots,:Z_4] \mapsto [\xi^a_0Z_0: \ldots:\xi^{a_4}Z_4]
\]
and this action preserves $X_{\psi}$.  Taking the quotient yields ($\psi^5 \ne 1$) an orbifold $X_{\psi}/G$.  When $\psi^5=1$, $X_{\psi}$ has $125$ disjoint nodal singularities which are permuted by $G$.  Thus, for $\psi^5=1$, $X_{\psi}/G$ is a Calabi-Yau orbifold with one nodal singularity.  It turns out \cite{Morrison2} that we can resolve the orbifold singularities of $X_{\psi}/G$ (simultaneously, for all $\psi$) and by doing so one obtains a family, the ``mirror quintic family"
\[
Y_{\psi} = \widetilde{X_{\psi}/G}
\]
where $\psi$ is smooth for $\psi^5\ne 1$, and has one node for $\psi^5=1$.  The Hodge diamond of the mirror quintic is given by
\begin{equation}\label{eq: HDMQ}
\begin{array}{ccccccc}
\,&\,&\,&1&\,&\,&\, \\
\,&\,&0&\,&0&\,&\, \\
\,&0&\,&101&\,&0&\, \\
1&\,&1&\,&1&\,&1\\
\,&0&\,&101&\,&0&\,\\
\,&\,&0&\,&0&\,&\,\\
\,&\,&\,&1&\,&\,&,
\end{array}
\end{equation}
Resolving the single nodal singularity of $Y_{1}$, we obtain a conifold transition
\[
\widehat{Y}_{1}\rightarrow Y_1 \leadsto Y_{\psi}
\]
As in the example of the generic quintic, the exceptional curve of the small resolution $\widehat{Y}_{1}\rightarrow Y_1$ is homologically trivial, by Theorem~\ref{thm: Friedman}, and hence $\widehat{Y}_1$ is non-K\"ahler.  Furthermore, by Proposition~\ref{prop: topChange}, the Hodge diamond of $\widehat{Y}_1$ is
\begin{equation}\label{eq: HDMQrigid}
\begin{array}{ccccccc}
\,&\,&\,&1&\,&\,&\, \\
\,&\,&0&\,&0&\,&\, \\
\,&0&\,&101&\,&0&\, \\
1&\,&0&\,&0&\,&1\\
\,&0&\,&101&\,&0&\,\\
\,&\,&0&\,&0&\,&\,\\
\,&\,&\,&1&\,&\,&,
\end{array}
\end{equation}
and so $\widehat{Y}_1$ is a rigid, non-K\"ahler Calabi-Yau $3$-fold, with $h^{1,1}=101$.

\subsubsection{The Tian-Yau Example}\label{subsec: LuTian}

The following manifold was first considered by Tian-Yau \cite{TianYau2}, and subsequently used by Lu-Tian \cite{LuTian} to construct an interesting conifold transition.  Define
\[
\begin{aligned}
\Gamma_1 &= \{\sum_{i=0}^{3} x_i^3 =0\} \subset \mathbb{P}^3\\
\Gamma_2 &= \{\sum_{i=0}^{3} y_i^3 =0\} \subset \mathbb{P}^3\\
H&= \{\sum_{i=0}^{3} x_iy_i=0\} \subset \mathbb{P}^3\times \mathbb{P}^3
\end{aligned}
\]
and let
\[
\widehat{X}:= (\Gamma_1 \times \Gamma_2) \cap H \subset \mathbb{P}^3 \times \mathbb{P}^3.
\]
$\widehat{X}$ is a smooth, simply connected Calabi-Yau $3$-fold, with Hodge diamond
\begin{equation}\label{eq: LuTianHodge}
\begin{array}{ccccccc}
\,&\,&\,&1&\,&\,&\, \\
\,&\,&0&\,&0&\,&\, \\
\,&0&\,&14&\,&0&\, \\
1&\,&23&\,&23&\,&1\\
\,&0&\,&14&\,&0&\,\\
\,&\,&0&\,&0&\,&\,\\
\,&\,&\,&1&\,&\,&,
\end{array}
\end{equation}
One can find (explicitly) $14$ disjoint $(-1,-1)$ rational curves $C_1,\ldots, C_{14}$ such that the homology classes $[C_i]$ span $H_2(M,\mathbb{C})$.  Furthermore, there exists a further $(-1,-1)$ rational curve $\gamma$, disjoint from $C_1,\ldots, C_{14}$ and such that
\[
[\gamma] = \sum_{i=1}^{14}\lambda_i [C_i] \quad \lambda_i \in \mathbb{C}^*
\]
for some $\lambda_i \in \mathbb{C}^*$. We can contract the 15 rational curves $C_1,\ldots, C_14, \gamma$ to obtain a nodal Calabi-Yau with $15$ ODP singularities, which is smoothable by Friedman's theorem.  This yields a conifold transition $\widehat{X}\rightarrow X_0 \leadsto X_t$.  By Proposition~\ref{prop: topChange}, $X_t$ has Hodge diamond
\[
\begin{array}{ccccccc}
\,&\,&\,&1&\,&\,&\, \\
\,&\,&0&\,&0&\,&\, \\
\,&0&\,&0&\,&0&\, \\
1&\,&24&\,&24&\,&1\\
\,&0&\,&0&\,&0&\,\\
\,&\,&0&\,&0&\,&\,\\
\,&\,&\,&1&\,&\,&\,
\end{array}
\]
By Wall's classification theorem, $X_t$ is diffeomorphic to $\#_{25}(S^3\times S^3)$.  In fact, subsequent work of Lu-Tian \cite{LuTian2} finds complex structures on $\#_{k}(S^3\times S^3)$ for all $k \geq 2$.

The above examples show two important phenomena:
\begin{itemize}
\item The property of a Calabi-Yau threefolds being K\"ahler is rather fragile.
\item By considering elementary examples we can construct many interesting Calabi-Yau threefolds with different topological types using conifold transitions.
\end{itemize}

\subsection{Mirror symmetry}

Mirror symmetry is a mysterious duality between different Calabi-Yau manifolds arising from string theory.  Mirror symmetry refers the ``symmetry" between the Hodge diamonds of mirror dual Calabi-Yau manifolds.  For Calabi-Yau threefolds this amounts to the statement that if $X, \check{X}$ are mirror dual Calabi-Yau threefolds, then
\[
h^{1,1}(X) = h^{2,1}(\check{X}) \qquad h^{2,1}(X) = h^{1,1}(\check{X}).
\]
Mirror symmetry for projective Calabi-Yau manifolds with $h^{2,1}(X)>0$ has been the subject of intense research over the past 30 years; see for example \cite{ClayMS, SYZ}.   However, for non-K\"ahler, or rigid Calabi-Yau manifolds, the situation is still rather mysterious.  These two situations are related, since if $X$ is a rigid Calabi-Yau manifold, so that $h^{2,1}(X)=0$, then necessarily the ``mirror manifold", $\check{X}$ must be non-K\"ahler.  One of the early applications of conifold transitions, as suggested in the physics literature (see e.g. \cite{GMS}), was to extend mirror symmetry to some non-K\"ahler and rigid Calabi-Yau manifolds.  The physics conjecture, due to Morrison \cite{Morrison}, states roughly that mirror symmetry reverses conifold transitions.  Namely, suppose
\[
X \rightarrow X_0 \leadsto Z
\]
is a conifold transition, and suppose that $\check{X}, \check{Z}$ are mirror to $X,Z$ respectively.  Then Morrison's conjecture asserts that $\check{X}$ and $\check{Z}$ are connected by a conifold transition
\[
\check{Z} \rightarrow Y_0 \leadsto \check{X}
\]
Let us pursue this idea to see what it might imply, particularly with respect to mirror symmetry for rigid and non-K\"ahler Calabi-Yau manifolds

\subsubsection{The Mirror Quintic, again}

Let us consider the example of the mirror quintic family $Y_{\psi}$ discussed in section~\ref{subsec: mirrorQuintic}.  Consider the degeneration
\[
Y_{\psi} \rightarrow Y_{1}
\]
where $Y_1$ has a single node.  Resolving this node yields a conifold transition
\[
\widehat{Y}_1\rightarrow Y_{1} \leadsto Y_{\psi}
\]
where the Hodge diamonds of $\widehat{Y}_{1}$, and $Y_{\psi}$ are given in~\eqref{eq: HDMQrigid} and ~\eqref{eq: HDMQ} respectively.  Our goal is to understand the mirror of $\widehat{Y}_1$, which we recall is non-K\"ahler, and rigid, in the sense the $h^{2,1}=0$.  Let us denote this manifold by $Z$.   Assuming Morrison's conjecture, $Z$ should satisfy
\[
X_{\psi}\rightarrow V\leadsto Z
\]
for some $V$.  We need to determined the number $k$ of rational curves contracted by $X_{\psi}\rightarrow V$.  Suppose that $k \geq 2$.  Since $h^{1,1}(X_{\psi},\mathbb{R})=1$ if we contract $k \geq 2$ rational curves they must satisfy Friedman's relation (since $h^{1,1}(X_{\psi})=1$) and hence $V$ is smoothable.  By Proposition~\ref{prop: topChange} we have
\[
\begin{aligned}
h^{1,1}(Z) = h^{1,1}(X_{\psi})-1 =0\\
h^{2,1}(Z) = h^{2,1}(X_{\psi}) + k-1= 100+k.
\end{aligned}
\]
To be consistent with mirror symmetry for Hodge numbers we need $h^{2,1}(Z) = h^{1,1}(\widehat{Y}_{1})=101$, and so $k=1$.  However, if we contract $k=1$ rational curves then $V$ is not smoothable, by Friedman's theorem.  In this case $b_2(Z)=0$, so $Z$ is not symplectic and the vanishing cycle $S^3\subset Z$ is homologically trivial. By Wall's classification \cite{Wall}, $Z= \# 102 (S^3\times S^3)$.  Note that Lu-Tian's result \cite{LuTian, LuTian2} implies that $Z$ admits a complex structure, but it is unclear whether this is the ``correct" complex structure, for the purposes of mirror symmetry.  In any case, $Z$ does not have a complex degeneration to $V$.


\section{The Web Of Calabi-Yau threefolds}\label{sec: web}

Simply connected, K\"ahler Calabi-Yau threefolds do not form a connected moduli space.  For example, the Fermat quintic considered in Section~\ref{sec: CY3}, and the Tian-Yau manifold of section~\ref {subsec: LuTian} are topologically distinct, as can be seen from their Hodge diamonds; see Lemma~\ref{lem: HodgeDiamondQuintic} and ~\eqref{eq: LuTianHodge}.  Reid \cite{Reid}, inspired in part by the work of Clemens \cite{Clemens} and Friedman \cite{Fri86}, has speculated that the ``the moduli space of $3$-folds with $K_X=0$ may nevertheless be irreducible".  This speculation, which has come to be known as Reid's Fantasy, is based on the idea that allowing conifold transitions and non-K\"ahler Calabi-Yau 3-folds to appear in our ``moduli space", we may pass between Calabi-Yau threefolds of different topological type.  Some motivation for this idea is provided by considering the classical case of moduli of $K3$ surfaces. 

\subsection{The moduli of $K3$ surfaces}

Consider the moduli space of {\em algebraic} $K3$ surfaces.  As we saw in Section~\ref{sec: CY3}, the moduli space of quartic hypersurfaces in $\mathbb{P}^3$ is $19$-dimensional.  On the other hand, consider  the complete intersection $K3$ surface
\[
X=\{P_2(Z_0,\ldots, Z_4=0\} \cap \{P_3(Z_0,\ldots, Z_4)=0\} \subset \mathbb{P}^4
\]
where $P_k$ are generic polynomials of degree $k=2,3$.  For generic choices $X$ is a smooth complete intersection, and $K_{X}\sim \mathcal{O}_{X}$ by adjunction. By the Lefschetz hyperplane theorem $X$ is a $K3$ surface.  Consider the line bundle $\mathcal{O}_{X}(1)$.  It is straightforward to compute that ${\rm deg}(\mathcal{O}_{X}(1))=6$.  On the other hand, it turns out that for a (very) general choice of the polynomials $P_2, P_3$, $X$ will have Picard rank $1$; this follows easily from the Torelli theorem, see e.g. \cite{HuybrechtsK3}.  Furthermore, $\mathcal{O}_{X}(1)$ is primitive since $6 \ne a^2 m$ for $a, m \in \mathbb{Z}_{>1}$. Thus, up to taking tensor powers, $\mathcal{O}_{X}(1)$ is the {\em only} line bundle on $X$; fix such a general choice of $X$.  We claim that $X$ cannot be embedded as a quartic in $\mathbb{P}^3$.  Suppose $X\hookrightarrow \mathbb{P}^3$.  Then $L= \iota^*\mathcal{O}_{\mathbb{P}^3}(1)$ is a line bundle  on $X$ and $L^2= 4$. But since $X$ has Picard rank $1$, and $\mathcal{O}_{X}(1)$ is primitive, this is impossible.

Now suppose we are interested in studying the moduli space of {\em algebraic} $K3$ surfaces.  We have seen that
\begin{itemize}
\item The moduli space of quartic hypersurfaces on $\mathbb{P}^3$ is $19$ dimensional.  
\item The moduli space of algebraic $K3$-surfaces containing the complete intersection of degree $(2,3)$ in $\mathbb{P}^4$ is not contained in the moduli space of quartic hypersurfaces in $\mathbb{P}^3$.
\end{itemize}

From these two examples we see that, at best, the moduli space of algebraic $K3$-surfaces is reducible.  In fact, by the Torelli theorem \cite{HuybrechtsK3}, algebraic $K3$-surfaces always lie in a $19$-dimensional moduli space of algebraic $K3$ surfaces.   The moduli space of algebraic $K3$ surfaces therefore appears to be extremely complicated, involving many different components possibly intersecting along lower dimensional strata.

The picture is significantly clarified by expanding our notion of moduli to include non-algebraic $K3$ surfaces.  If we adopt this point of view then Kodaira showed that the moduli space is a smooth manifold of complex dimension $20$.  The chaotic nature of the moduli of algebraic $K3$ surfaces reflects co-dimension $1$ phenomena in this larger moduli space.

\subsection{Reid's Fantasy for Calabi-Yau threefolds}

Consider the moduli space of K\"ahler Calabi-Yau threefolds.  As we have seen, this moduli space is not connected and contains representatives with different topological type.  On the other hand, conifold transitions allow us to pass between Calabi-Yau threefolds with different topological type. Following Reid we may ``fantasize" that the chaotic nature of the moduli of K\"ahler Calabi-Yau threefolds is due to their appearance as some lower dimensional subspace in a larger, more well-behaved moduli space of (not necessarily K\"ahler) Calabi-Yau threefolds. Below we give a formulation of Reid's Fantasy follow work of Gross \cite{Gross1, Gross2}.

Define a directed graph of Calabi-Yau threeolds as follows:
\begin{itemize}
\item A node in the graph corresponds to a deformation family of smooth, compact Calabi-Yau threefolds.
\item Given two nodes $\mathcal{M}_1, \mathcal{M}_2$ we draw an arrow
\[
\mathcal{M}_1 \rightarrow \mathcal{M}_2
\]
if, for the general member $X \in \mathcal{M}_1$ there is a conifold transition $\mathcal{M}_1 \ni X \rightarrow X_0 \leadsto X_t \in \mathcal{M}_2$.
\end{itemize}

\begin{conj}[Reid's Fantasy \cite{Reid}]
The graph of simply connected Calabi-Yau threefolds is connected.
\end{conj}

We can further expand the notion of a conifold transition to a {\em geometric transition}, to obtain a weaker version of Reid's conjecture. Generally speaking, a geometric transition consists of a birational contraction $\pi:X \rightarrow X_0$, followed by a smoothing $X_0 \leadsto X_t$.  General geometric transitions for Calabi-Yau threefolds are well-studied in the mathematics literature, but are beyond the scope of these lecture notes.  We refer the reader to \cite{Rossi, Gross1, Gross2} and the references therein.


\subsection{Evidence for Reid's Fantasy}

The evidence for Reid's Fantasy is primarily experimental.  We refer the reader to the work of Green-H\"ubsch \cite{GH1, GH2}, Candelas-Green-H\"ubsch \cite{CGH}, Chiang-Greene-Gross-Kanter \cite{CGGK}, and more recently the work of Wang \cite{Wang}.  The main result of Green-H\"ubsch is the following

\begin{thm}[Green-H\"ubsch \cite{GH1}, Wang \cite{Wang}]
Any two complete intersection Calabi-Yau threefolds in a product of projective spaces are connected by a finite sequence of conifold transitions.
\end{thm}

Chiang-Greene-Gross-Kanter \cite{CGGK} studied the connectedness of Calabi-Yau complete intersections in toric varieties and described a general algorithm for determining whether these threefolds are connected by general geometric transitions.   This algorithm was applied to verify that all Calabi-Yau hypersurfaces in weighted projective four space are mathematically connected.  


\subsection{The vacuum degeneracy problem}

There are only four consistent string theories in $10$-dimensions: the type IIA/B theories, and two types of heterotic string theory.  In order to get a theory in four dimensions, one assumes that $6$ of the $10$ dimensions are compactified to be extremely small; that is, our $10$-dimensional space is of the form $\mathbb{R}^{1,3} \times X $ for some compact ``internal" $6$ manifold $X$.  If one assumes the theory to have no ``flux", then the internal space $X$ is Calabi-Yau \cite{CHSW}.  Unfortunately (or fortunately), compact Calabi-Yau $6$ manifolds are plentiful, thanks to Yau's theorem \cite{Yau78}.  This fact limits the predictive power of string theory, since in order to calculate some physical quantity, one needs to make a choice of the internal manifold $X$.  Even placing phenomenological restrictions on the Calabi-Yau manifold $X$ does not lead to a unique vacuum configuration; see e.g. \cite{CdlOHS}.  Green-H\"ubsch \cite{GH1, GH2} and Candelas-Green-H\"ubsch \cite{CGH} pioneered the idea that conifold transitions could unify string vacua through topology changing transitions.  Strominger \cite{Stromingerholes}, and Greene-Morrison-Strominger \cite{GMS} showed the for type II theories, conifold transitions could be made continuous at the level of string physics.

\section{Metric Aspects of Conifold Transitions}\label{sec: metric}

In this section we will describe the construction of explicit Ricci-flat K\"ahler metrics through a conifold transition.

\subsection{Ricci-flat K\"ahler metrics on the deformation family}
We construct an explicit family of metrics on the deformed conifold, which were discovered independently by Candelas- de la Ossa and Stenzel.
\[
V_t = \{ \sum_{i=1}^4 z_i^2=t\}
\]
for $t \in \mathbb{C}$.  Since $V_t$ is Stein, it has no nontrivial cohomology and so we may as well look for an exact Calabi-Yau metric.  That is, we look for a function $\phi_t: V_t \rightarrow \mathbb{R}$ such that
\[
\omega_{co,t} = \ddbar \phi_t >0, \qquad \omega_{co,t}^3 = \sqrt{-1}^{3^2}\Omega_t \wedge \overline{\Omega_t}.
\]
Observe that $V_t$ admits an action by $SO(4,\mathbb{C})$.  It is therefore natural to look for a function $\phi_t$ that is invariant under the compact real form $SO(4,\mathbb{R}) \subset SO(4,\mathbb{C})$.  The function
\[
\tau(z) = \|z\|^2
\]
is invariant under $SO(4, \mathbb{R})$, and hence we consider the ansatz
\begin{equation}\label{eq:  defAnsatz}
\phi_t(z) = f_t(\tau(z)).
\end{equation}
\begin{lem}
Under the ansatz~\eqref{eq: defAnsatz}, $\omega_{co,t}$ solves the Monge-Amp\`ere equation if and only if $f:=f_t$ satisfies
\begin{equation}\label{eq: cotODE}
\begin{aligned}
&\frac{df}{d\tau}>0,\quad \frac{4\tau}{|t|+\tau} \frac{df}{d\tau} + (\tau^2-|t|)\frac{d^2f}{d\tau^2} >0 \qquad \\\ 
&\left(\frac{df}{d\tau}\right)^3\tau + \left(\frac{df}{d\tau}\right)^2\frac{d^2f}{d\tau^2}\left(\tau^2-|t|^2\right) = c
\end{aligned}
\end{equation}
for $c \in \mathbb{R}_{>0}$.
\end{lem}

\begin{proof}
We will prove the result for $t=1$ and then deduce the general case using the rescaling action.  Fix $R\geq 1$.  Since $SO(4,\mathbb{R})$ acts transitively on $V_1 \cap\{\|z\|^2=R^2\}$ we can assume that
\[
z_1 = \sqrt{-1}\sqrt{\frac{(R^2-1)}{2}}, \quad z_2=z_3=0,\quad z_4= \sqrt{\frac{1+R^2}{2}}
\]
From the defining equation of $V_0$ we have
\[
dz_4 = -\frac{z_1dz_1}{z_4}  
\]
and so
\[
\begin{aligned}
\ddbar \tau &= \left(1+\frac{|z_1|^2}{|z_4|^2}\right) \sqrt{-1}dz_1\wedge d\bar{z}_1 +  \sqrt{-1}dz_2\wedge d\bar{z}_2 + \sqrt{-1}dz_3\wedge d\bar{z}_3 \\
\sqrt{-1}\del \tau \wedge \dbar{\tau} &= (\bar{z}_1 - \frac{\bar{z}_4z_1}{z_4})(z_1 - \frac{z_4\bar{z}_1}{\bar{z}_4})\sqrt{-1}dz_1\wedge dz_1 \wedge 
\end{aligned}
\]
since $z_4 \in \mathbb{R}$ and $z_1 \in \sqrt{-1}\mathbb{R}$ we arrive at
\[
\sqrt{-1}\del \tau \wedge \dbar{\tau} = 4|z_1|^2\sqrt{-1}dz_1\wedge dz_1 \wedge. 
\]
We now compute
\[
\begin{aligned}
\ddbar f(\tau) &= \frac{d f}{d\tau} \ddbar \tau + \frac{d^2 f}{d\tau^2} \sqrt{-1}\del \tau \wedge \dbar \tau\\
&=\left( 4|z_1|^2\frac{d^2 f}{d\tau^2} +\left(1+\frac{|z_1|^2}{|z_4|^2}\right)\frac{d f}{d\tau}\right) \sqrt{-1}dz_1 \wedge d\bar{z}_1\\
& +\frac{df}{d\tau} \left( \sqrt{-1}dz_2\wedge d\bar{z}_2 + \sqrt{-1}dz_3\wedge d\bar{z}_3\right) \\
&=\left(2(\tau-1) \frac{d^2 f}{d\tau^2} + \left(\frac{2\tau}{1+\tau} \frac{d f}{d\tau}\right)\right)\sqrt{-1}dz_1 \wedge d\bar{z}_1\\
&+\frac{df}{d\tau} \left( \sqrt{-1}dz_2\wedge d\bar{z}_2 + \sqrt{-1}dz_3\wedge d\bar{z}_3\right).
\end{aligned}
\]
This formula defines a metric provided $f$ satisfies the first two conditions of ~\eqref{eq: cotODE}  (for $t=1$).  Next we can compute the volume form.  
\[
\left(\ddbar f(\tau)\right)^3 =2\cdot 3! \left(\frac{df}{d\tau}\right)^2\left((\tau-1) \frac{d^2 f}{d\tau^2} + \left(\frac{\tau}{1+\tau} \frac{d f}{d\tau}\right)\right)i^{3}  dz_1d\bar{z}_1dz_2d\bar{z}_2dz_3 d\bar{z}_3
\]
where we suppressed the wedge products.  In order to solve the complex Monge-Amp\`ere equation, we need
\[
\frac{\left(\ddbar f(\tau)\right)^3}{3!} =c (\sqrt{-1})^{3^2}\Omega \wedge\overline{\Omega}.
\]
This yields the equation
\[
 \left(\frac{df}{d\tau}\right)^2\left((\tau-1) \frac{d^2 f}{d\tau^2} + \left(\frac{4\tau}{1+\tau} \frac{d f}{d\tau}\right)\right) = \frac{c}{1+\tau}.
 \]
 Rearranging gives
 \[
  \left(\frac{df}{d\tau}\right)^2\left((\tau^2-1) \frac{d^2 f}{d\tau^2} + \tau\frac{d f}{d\tau}\right)  = c
  \]
  which is the desired result for $t=1$.
  
 For general $t$ we consider the rescaling action $S_{t^{-1/3}}: V_{t}\rightarrow V_1$.  Then, $S_{t^{-1/3}}^*\Omega_1= t\Omega_t$  and so  $f_t(\tau) = |t|^{-2/3}f_1(|t|^{-1}\tau)$ solves ~\eqref{eq: cotODE}.  
  \end{proof}

The particular choice of positive constant $c$ in~\eqref{eq: cotODE} is irrelevant; different choices of constant correspond to an overall scaling of the metric. Let us analyze the solution to the equation~\eqref{eq: cotODE}.  We consider first the case $t=0$, and make the convenient choice of constant $c=\frac{1}{6}$.  so that ~\eqref{eq: cotODE} reduces to
\[
\left(\frac{df}{d\tau}\right)^3\tau + \left(\frac{df}{d\tau}\right)^2\frac{d^2f}{d\tau^2}\tau^2= \frac{1}{6}
\]
 We make the change of variables $\tau= s^2$ and write
\[
\gamma(s) = s^2f'(s^2)
\]
where $f' = \frac{df}{d\tau}$.  Then
\[
\begin{aligned}
\frac{d}{ds} \gamma^3  &= 3\gamma^2 (2sf'(s^2) + 2s^3 f''(s^2))\\
& = 6s^3(s^2\left(f'(s^2)\right)^3 + s^4 \left(f'(s^2)\right)^2f''(s^2))\\
&=s^3
\end{aligned}
\]
and so $\gamma(s) = \left(\frac{s^4}{4}\right)^{1/3}$.  In other words $\tau \frac{df}{d\tau} = \left(\frac{\tau^2}{4}\right)^{1/3}$.
This equation can be integrated directly to obtain
\[
f(\tau) = \frac{3}{2\cdot 4^{1/3}}\tau^{2/3}.
\]
After rescaling we have
\begin{lem}
The metric $\omega_{co,0} = \ddbar \|z\|^{4/3}$ is an explicit Ricci flat metric on the conifold $V_0 =\{\sum_{i=1}^{4} z_i^2=0\}$.
\end{lem}

\begin{exercise}
Show that the Riemannian structure associated with $\omega_{co,0}$ is a metric cone.  That is, show that there is a function $r: V_0\rightarrow \mathbb{R}_{\geq 0}$ such that, on $V_0\setminus\{0\}$ we have
\[
dr^2+r^2g_{L}
\]
where $g_{L}$ is an Einstein metric with positive Ricci curvature on $L = V_0 \cap \{r=1\}$. 
\end{exercise}

Now let us consider the case of ~\eqref{eq: cotODE} in the case $t \ne 0$.  Using the rescaling action $S_{t^{-1/3}}:V_t \rightarrow V_1$ we can reduce to the case $t=1$, and let us write $f_1=f$ for simplicity.  Then, if we let $\mu(\tau) = \sqrt{\tau^2-1}$, ~\eqref{eq: cotODE} can be written as
\[
\left(\frac{df}{d\tau}\right)^2 \mu(\tau) \frac{d}{d\tau}\left(\mu(\tau)\frac{df}{d\tau}\right) = c  
\] 
or,
\[
 \frac{d}{d\tau}\left(\mu(\tau)\frac{df}{d\tau}\right)^3 = 3c\mu
 \]
This yields
\[
\begin{aligned}
\left(\mu(\tau)\frac{df}{d\tau}\right)^3 &= 3c \int \mu d\tau \\
&= \frac{3c}{2} \left(\tau \sqrt{\tau^2-1} -\log\left(\tau + \sqrt{\tau^2-1}\right)\right).
\end{aligned}
\]
Take $c=2/3$ for simplicity.  Solving for $\frac{df}{d\tau}$ yields
\begin{equation}\label{eq: cotfprime}
\frac{df}{d\tau} = \frac{1}{\sqrt{\tau^2-1}}\left(\tau \sqrt{\tau^2-1} -\log\left(\tau + \sqrt{\tau^2-1}\right)\right)^{\frac{1}{3}}.
\end{equation}
To integrate this expression we introduce $\lambda = \cosh^{-1}(\tau)$, so that $d\lambda =\frac{d\tau}{\sqrt{\tau^2-1}}$.  Then
\[
\begin{aligned}
\tau \sqrt{\tau^2-1} &=\cosh(\lambda)\sinh(\lambda) = \frac{1}{2}\sinh(2\lambda)\\
\log\left(\tau + \sqrt{\tau^2-1}\right)& = \log\left(\cosh(\lambda)+ \sinh(\lambda)\right) =\lambda
\end{aligned}
\]
so in the end we get
\[
f_1(\tau)=2^{-1/3}\int_0^{\cosh^{-1}(\tau)}\left(\sinh(2\lambda)-2\lambda\right)^{\frac{1}{3}} d\lambda.
\]
In general, we get

\begin{lem}\label{lem: cotMetric}
The function $f_t(\tau)$ solving~\eqref{eq: cotODE} is, up to an additive constant, given by
\[
f_{t}(\tau) = |t|^{-2/3}2^{-1/3}\int_0^{\cosh^{-1}(\frac{\tau}{|t|})}\left(\sinh(2\lambda)-2\lambda\right)^{\frac{1}{3}} d\lambda.
\]
For $\tau \gg |t|$, $f_t$ has an expansion
\[
f_t(\tau) = \frac{3/2}\tau^{2/3} + \tilde{c}_1 |t|^{2/3}\tau^{-4/3}\log(\frac{\tau}{|t|}) +\tilde{c}_2 |t|^{5/3}\tau^{-7/3} + o(|t|^{7/3}\tau^{-3}).
\]
\end{lem}

The only thing we have not established is the asymptotics of $f_t$ as $\tau \rightarrow +\infty$.  This turns out to be more straightforward if we use~\eqref{eq: cotfprime} rather than the above expression for $f_t$.  Again, we consider only the case $f=f_1$, and obtain the general case by rescaling.  Expanding $\frac{df}{dt}$ for $t\gg 1$ we have
\[
\frac{df}{d\tau} = \tau^{-1/3} + c_1 \tau^{-7/3}\log(\tau) + c_2\tau^{-7/3} + O(\tau^{-13/3}\log(\tau)^2).
\]
Upon integration this yields the following estimate, which is somewhat wasteful in the error terms.
\[
f(\tau) = \frac{3}{2}\tau^{2/3} + \tilde{c}_1 \tau^{-4/3}\log(\tau) + \tilde{c}_2\tau^{-7/3} + o(\tau^{-3}).
\]
The term $\frac{3}{2}\tau^{2/3}$ encodes the conical Calabi-Yau metric on on the conifold $V_0$, while the lower order terms decay.  This shows that the metric $\omega_{co,1}$ is asymptotically conical with tangent cone at infinity being $V_0$ with the Ricci-flat K\"ahler metric $\omega_{co,0}$.  To make this completely rigorous one can use the map $\Phi_t$ defined in~\eqref{eq: PhitMap} to identify $\omega_{co,t}$ with a family of metrics on $V_0$ defined outside $\{\|z|^2 >\frac{|t|}{2}\}$.

Finally, let us remark that the $S^3 \subset V_t$ given by $\|z\|^2=t$ is clearly  Lagrangian with respect to the Calabi-Yau structure given by $\omega_{co,t}$, and so, by Lemma~\ref{lem: sLagMod} it is {\em special Lagrangian}.

\subsection{Ricci-flat K\"ahler metrics on the small resolution}

We now consider the manifold $\widehat{V}:= \mathcal{O}_{\mathbb{P}^1}(-1)^{\oplus 2}$.  This manifold contains a compact complex curve, given by the zero section of the bundle.  In particular, in order to construct a K\"ahler metric we must fix a choice of a K\"ahler form.  The natural choice is to take $\pi^*\omega_{FS}$, where $\omega_{FS}$ is the Fubini-Study K\"ahler form on $\mathbb{P}^1$ and $\pi: \widehat{V}\rightarrow \mathbb{P}^1$ is the natural projection.  We look for a Calabi-Yau metric of the form
\[
\omega_{co,a} = 4a^2 \pi^*\omega_{FS} + \ddbar \phi_a
\]
We are going to apply the same philosophy of symmetry reduction employed in the previous section. There is a natural symmetry group generated by the actions
\[
\begin{aligned}
(\mathbb{C}^*)^2 \ni (\lambda_1, \lambda_2) &\mapsto \{ (\xi_1, \xi_2) \mapsto (\lambda_1 \xi_1, \lambda_2 \xi_2)\}\\
\mathbb{Z}_{2} \ni (-1) &\mapsto \{ (\xi_1, \xi_2) \mapsto ( \xi_2, \xi_1)\}
\end{aligned}
\]
Let $G$ be the group of automorphisms generated.  Let $[X_1:X_2]\in \mathbb{P}^1$, and denote by $h_{FS} = |X_1|^2 + |X_2|^2$ denote the Fubini-Study metric on $\mathcal{O}_{\mathbb{P}^1}(-1)$.  This choice of metric splits $\mathbb{C}^* = S^1\times \mathbb{R}_{>0}$.  We look for a function $\phi$ which is invariant under the action of the compact group generate by $\mathbb{Z}_2$ and $(S^1)^2 \subset (\mathbb{C}^*)^2$.  Such a metric must have a potential of the form
\begin{equation}\label{eq: coaAnsatz}
\phi_a= f_a(\tau) \qquad \tau = (|X_1|^2+|X_2|^2)(|W_1|^2+|W_2|^2).
\end{equation}

We also remark that, by rescaling
\[
\omega_{co,a} = a^2S_{a^{-1}}^*\omega_{co,1} = 4\pi^*\omega_{FS} + \ddbar \left(a^2 f_1\left(\frac{\tau}{a^3}\right)\right)
\]
and hence we can take $f_a = a^2 f_1\left(\frac{\tau}{a^3}\right)$.  

\begin{lem}
Under the ansatz~\eqref{eq: coaAnsatz}, the metric $\omega_{co,a}$ is Calabi-Yau if $f_a= f_a(\tau)$ satisfies
\[
\begin{aligned}
&\frac{df_a}{d\tau}>0, \quad \frac{df_a}{d\tau} +\tau \frac{d^2f_a}{d\tau^2} >0\\
&(4a^2 + \tau\frac{df_a}{d\tau})\left(\left(\frac{df_a}{d\tau}\right)^2 + \tau\frac{df_a}{d\tau}\frac{d^2f_a}{d\tau^2}\right)= c
\end{aligned}
\]
for $c \in \mathbb{R}_{>0}$ a constant
\end{lem}
\begin{proof}
Suppose we work on the patch $\{X_1=1\}$, the other case being identical.  Furthermore, by a linear transformation we may assume $X_2=0$.  Writing $f=f_a$ for simplicity  we compute
\[
\ddbar f = f' \ddbar\tau + f'' \sqrt{-1}\del\tau \wedge \dbar{\tau}.
\]
Since we are working at the point $(1,0)$ we have
\[
\begin{aligned}
\tau &=|W_1|^2+|W_2|^2\\
\del \tau &= \overline{W}_1dW_1 + \overline{W}_2 dW_2\\
\ddbar \tau &= \pi^*\omega_{FS} (|W_1|^2+|W_2|^2) + \sqrt{-1}(dW_1\wedge d\overline{W}_1 + dW_2\wedge d\overline{W}_2).
\end{aligned} 
\]
It follows that
\[
\begin{aligned}
\sqrt{-1}\del \tau \wedge \dbar \tau &= |W_1|^2 \sqrt{-1}dW_1\wedge d\overline{W}_1 + |W_2|^2\sqrt{-1}dW_2 \wedge d\overline{W}_2 \\
&+ 2{\rm Re}\left(\sqrt{-1}W_2\overline{W}_1 dW_1\wedge d\overline{W}_2 +  W_1 \overline{W}_2dW_2\wedge d\overline{W}_1 \right)
\end{aligned}
\]
and so
\[
\begin{aligned}
\omega_{co,a} &= (4a^2 + f'\tau)\pi^*\omega_{FS} + f'\sqrt{-1}(dW_1\wedge d\overline{W}_1 + dW_2\wedge d\overline{W}_2\\
&+ f''\sqrt{-1}\del\tau \wedge \dbar{\tau}.
\end{aligned}
\]
Computing the wedge product yields
\[
\omega_{co,a}^3 = 3!(4a^2 + f'\tau)\left((f')^2 +f'' f'\tau\right) \pi^*\omega_{FS}\wedge \sqrt{-1}(dW_1\wedge d\overline{W}_1\wedge \sqrt{-1}dW_2\wedge d\overline{W}_2\
\]
This metric will be Calabi-Yau if
\[
(4a^2 + f'\tau)\left((f')^2 +f'' f'\tau\right)= c
\]
for $c\in \mathbb{R}>0$ a constant.
\end{proof}

Let us now investigate the solution of this equation.  Let $\gamma = \tau f'(\tau)$.  Then the equation can be rewritten as solves
\[
(4a^2 + \gamma )\gamma \gamma' = c\tau
\]
which admits a first integral
\[
2a^2\gamma^2+\frac{1}{3}\gamma^3=\frac{c}{2}\tau^2
\]
Choosing $c=\frac{2}{3}$ yields
\[
6a^2\gamma^2+\gamma^3=\tau^2
\]
As before, using the rescaling action we may assume that $a=1$.  This equation admits the solution
\[
\gamma(\tau) = -2 + z+ \frac{4}{z}, \qquad z= 2^{-1/3}(-16+\tau^2+\sqrt{-32\tau^2+\tau^4})^{1/3}
\]
We remark that the function $z(\tau)$ becomes complex when $\tau$ becomes small.  However, one can check that the expression for $\gamma(\tau)$ remains well-defined.  Indeed, if $z$ is complex then
\[
|z|^2 = 2^{-2/3}(256)^{\frac{1}{3}} = 4
\]
and so, for $\tau^2<32$ we have
\[
\gamma(\tau) = -2 + z+ \frac{4}{z}= -2+z+4\frac{\bar{z}}{|z|^2} = -2+z+\bar{z} = -2+2{\rm Re}(z).
\]
Thus, in general we have
\[
f_1(\tau) = \int_0^{\tau} \frac{1}{s} \gamma(s) ds.
\]
and one can check that, as $\tau \rightarrow 0$, $f_1' \rightarrow \frac{1}{\sqrt{6}}$.

Let us now extract the leading order asymptotics for $f$ as $\tau \rightarrow +\infty$. We do this by extracting the leading order asymptotics of $\tau^{-1}\gamma(\tau)$ and then integrating term by term.  We have
\[
\tau^{-1}\gamma(\tau)= \tau^{-1/3} -\frac{2}{\tau} +4\tau^{-5/4} + O(\tau^{-7/3})
\]
and so
\[
f_1(\tau) = \frac{3}{2}\tau^{2/3} -2\log(\tau) + O(\tau^{-1/4})
\]
Again we see that the leading order behvaior is given by the function $\frac{3}{2}\tau^{2/3}$, which defines the Calabi-Yau metric on $V_0$.  On the other hand, unlike the case of the of smoothed conifold, the subleading order term in the expansion of $f_1$ does not decay.   Summarizing, we have
\begin{lem}\label{lem: coaMetric}
Under the ansatz~\eqref{eq: coaAnsatz}, the Calabi-Yau metric on the resolved conifold, lying in the cohomology class $4a^2\pi^*[\omega_{FS}] \in H^{1,1}(\widehat{V},\mathbb{R})$ is given by
\[
f_a(\tau) = a^2\int_0^{a^{-3}\tau} \frac{1}{s}\gamma(s) ds
\]
where
\[
\gamma(s) = -2 + z+ \frac{4}{z}, \qquad z= 2^{-1/3}(-16+s^2+\sqrt{-32s^2+s^4})^{1/3}.
\]
Furthermore, $f_a(\tau)$ admits an expansion for $\tau \gg a^3$
\[
f_a(\tau) =\frac{3}{2}\tau^{2/3} -2a^2\log(a^{-3}\tau) + O(a^2\tau^{-1/4}).
\]
In particular, we see that $f_{a}(\tau) \rightarrow \frac{3}{2}\tau^{2/3}$ smoothly on compact sets as $a\rightarrow 0$.
\end{lem}

An immediate consequence of Lemma~\ref{lem: cotMetric} and Lemma~\ref{lem: coaMetric} is that 
\[
(\widehat{V}, \omega_{co,a}) \xlongrightarrow{a\rightarrow 0} (V_0, \omega_{co,0}) \xlongleftarrow{t \rightarrow 0} (V_t, \omega_{co,t})
\]
where the convergence is in the sense of Gromov-Hasudorff.  In particular, we see that, as Calabi-Yau manifolds, the conifold transition is continuous.

\section{Geometrizing Calabi-Yau manifolds through conifold transitions}\label{sec: geomConifold}

A natural approach to understanding Reid's fantasy is to put some kind of canonical metric on the varieties appearing in the web of Calabi-Yau threefolds.  Existence of canonical differential-geometric structures can be a powerful tool for probing the algebraic and topological properties of the underlying space.  But what kind of canonical metric should we consider?  If $Y$ is a K\"ahler Calabi-Yau 3-folds then there is a clear candidate: any of the Ricci-flat K\"ahler metrics produced by Yau's theorem (see Theorem~\ref{thm: Yau}).  On the other hand, as we have seen, Reid's fantasy necessitates passing to non-K\"ahler Calabi-Yau threefolds; how should we ``uniformize" these objects? 

There are many possible answers to this question.  One could look for hermitian metrics whose Chern connections have constant scalar curvature \cite{ACS}, or vanishing Chern-Ricci curvature \cite{STW,TW, TW2}, or for balanced metrics with vanishing Chern-Ricci curvature \cite{Fei2, GiustiSpotti}.  The point of view we shall pursue, as suggested by Yau, is to look for solutions of the heterotic string vacuum equations.  This is well motivated by the string theory literature and connections between Reid's Fantasy and the ``vacuum degeneracy problem" of string theory (see Remark~\ref{rk: heteroticVDP} below). 

The {\em heterotic string system}, or HS system is a set of equations for a Calabi-Yau threefold $X$ equipped with a non-vanishing holomorphic $(3,0)$ form $\Omega$ and a holomorphic vector bundle $E\rightarrow X$.  For string compactifications with zero flux, the vacuum equations were investigated in a celebrated paper of Candelas-Horowitz-Strominger-Witten \cite{CHSW}.  The case of string compactifications with flux was considered independently by Strominger \cite{Strominger} and Hull \cite{Hull}.  The equations of motion seek a hermitian metric $g$ on $T^{1,0}X$, with associated $(1,1)$-form $\omega$ and a hermitian metric $H$ on the gauge bundle $E$ such that  
\begin{equation} \label{eq: introHSbal}
d(\|\Omega\|_{\omega} \, \omega^2) =0,
\end{equation}
\begin{equation}\label{eq: introHSYM}
\omega^2 \wedge  F_{H}=0,
\end{equation}
\begin{equation}\label{eq: introAnom}
\ddbar \omega - \frac{\alpha'}{4}\left({\rm Tr} Rm_g\wedge Rm_g  - {\rm Tr} F_{H}\wedge F_{H}\right)=0.
\end{equation}
where $F_{H}$ denotes the curvature of the Chern connection of $(E,H)$, $\|\Omega\|^2_{\omega}$ is the norm of $\Omega$ with respect to $g$, and $\alpha'>0$ is a constant.  In the mathematics and physics literature there are several choices of connection which are commonly used to define the curvature $Rm_g$ in~\eqref{eq: introAnom} (see e.g. \cite{PicardMcOrist, Ivanov, LY05, GF-survey}). A common choice, and the one we shall adopt here, is to use the Chern connection (see \cite{PicardMcOrist} for recent working analyzing the compatibility of this choice with supersymmetry).  Equation~\eqref{eq: introHSbal} is called the ``conformally balanced equation", following the work of Michelsohn \cite{Michelsohn}. Equation~\eqref{eq: introHSYM}  is the familiar Hermitian-Yang-Mills equation.  Equation ~\eqref{eq: introAnom} is called the ``anomaly cancellation condition".  Clearly we must impose the conditions
\begin{equation}\label{eq: chernClassHS}
\begin{aligned}
c_1(E) &=0 \in H^{1,1}_{BC}(X, \mathbb{R}) \\
c_2(E) &= c_2(X) \in H^{2,2}_{BC}(X,\mathbb{R})
\end{aligned}
\end{equation}
as dictated by~\eqref{eq: introHSYM} and~\eqref{eq: introAnom} (here $H^{p,q}_{BC}$ denotes the Bott-Chern cohomology).

 The HS system is an extension of the K\"ahler Ricci-flat geometry of Yau's theorem to the non-K\"ahler setting.  Indeed, suppose $X$ is K\"ahler, and let $E=T^{1,0}X$.  Let $g$ be the Calabi-Yau metric produced by Yau, and let $H=g$.  Then, from the complex Monge-Amp\`ere equation we have $|\Omega|_{\omega}= {\rm const}$, so that~\eqref{eq: introHSbal} is equivalent to $d\omega \wedge \omega =0$, which is automatically satisfied thanks to the K\"ahler assumption.  Equation~\eqref{eq: introHSYM} is satisfied since
 \[
 \omega^2 \wedge F_{H} \propto {\rm Ric}(\omega) =0
 \]
Finally, if we take the Chern connection in the anomaly cancellation equation~\eqref{eq: introAnom}, then we have
\[
\ddbar \omega =0, \quad \left(\Tr Rm_g\wedge Rm_g  - \Tr F_{H}\wedge F_{H}\right)=0.
\]
since $Rm_g= F_H$. Thus the HS system can be viewed as providing a natural extension of the powerful theory of Calabi-Yau geometry to the non-K\"ahler context.  

There has recently been a great deal of interest in understanding the existence and uniqueness of solutions to the HS system.  The first solutions were constructed by Li-Yau \cite{LY05} as perturbations of K\"ahler Calabi-Yau solutions.  The first solutions on non-K\"ahler backgrounds were constructed by Fu-Yau \cite{FY}.  Further constructions on K\"ahler backgrounds were carried out in \cite{AGF, AGF2, CPY2}, and under various symmetry/fibration assumptions on the background geometry \cite{Fei, FHP, FIUV, FIUV2, FGV, FTY, OUV}.  The HS system also has deep and surprising connections with generalized geometry and the theory of string algebroids and higher gauge theory; see e.g. \cite{AGS, GRT, GRST, GM, GFMS}.  Parabolic approaches have been pioneered by Phong-Picard-Zhang \cite{PPZ18b, PPZ18a, PPZ18, PPZ19, FeiPhong, FeiPic}, based on the notion of an {\em Anomaly Flow}.  An alternative parabolic approach, based on extensions of the Streets-Tian pluriclosed flow to higher gauge theory has recently been pioneered by Garcia-Fernandez-Molina-Streets \cite{GFMS}.  We refer the reader to the survey articles \cite{GF-survey, PhongSurvey, PicardSurvey1, PicardLNS, PicardSurvey2}  and the references therein.


\begin{rk}\label{rk: heteroticVDP}
    The role of conifold transitions for resolving the ``vacuum degeneracy problem" is best understood in the setting of type II theories; see e.g. \cite{Stromingerholes, GMS}.  For theories with less supersymmetry, like the heterotic string, the situation is complicated by the presence of the gauge bundle.  For some recent progress towards understanding the role of conifold transitions in the unification of the heterotic string vacua, see \cite{CdlOHS, ABG} and the references therein.  One could therefore wonder whether it is more appropriate to use the equations of motion for the type IIA or type IIB string \cite{GMPT, Tomasiello, TsengYau} as a tool for geometrizing non-K\"ahler Calabi-Yau manifolds.  Entertaining this for the moment, we can immediately discard the type IIA string, since conifold transitions can produce Calabi-Yau manifolds with $b_2=0$, and hence no symplectic structure.  For the type IIB string, non-K\"ahler solutions necessarily have sources, which can be localized on calibrated submanifolds (D-branes and O-planes), or ``smeared".  We have essentially no non-trivial examples of compact manifolds with solutions of the type IIB equations with non-trivial sources.  In the few examples of type IIB backgrounds with sources that are understood, the background metric changes signature in an open neighborhood of the O-planes.    
\end{rk}

The existence of solutions to the HS system through conifold transitions is an area of active research. Roughly speaking, one would like to say that the Candelas-de la Ossa \cite{CdlO90} family of K\"ahler Ricci-flat metrics, as constructed in Section~\ref{sec: metric} describe the local geometry of solutions to the HS system through a conifold transition.  If one assumes that the smoothing $\mathcal{X}\rightarrow \Delta$ is projective, then a deep result of Hein-Sun \cite{HeinSun} says that integral K\"ahler-Ricci flat metrics $\omega_t$ on $X_t$ are quantitatively close to the Candelas-de la Ossa metrics near the special Lagrangian vanishing cycles.  We refer the reader also to \cite{ChiuSzek, XinFu, Song} for some related results.  We will focus primarily on the case when the smoothing is not necessarily K\"ahler.  In this direction, the first progress was made by Fu-Li-Yau \cite{FLY}, who solved the conformally balanced equation~\eqref{eq: introHSbal}.

\begin{thm}[Fu-Li-Yau \cite{FLY}]\label{thm: FLY}
Let $(\widehat{X}, \omega)$ be a K\"ahler Calabi-Yau threefold and $C_i \subset \widehat{X}$, $1 \leq i \leq k$ be a collection of disjoint $(-1,-1)$ rational curves satisfying Friedman's condition~\eqref{eq: FriedmanRelation}.  Let
\[
\widehat{X} \rightarrow X_0 \leadsto X_t
\]
be the conifold transition obtained by contracting the $C_i$  Then:
\begin{enumerate}
    \item For $a$ sufficiently small, there exist hermitian metrics $\omega_{FLY,a}$ on $\widehat{X}$ satisfying:
    \begin{itemize}
    \item[(i)] $d\omega_{FLY,a}^2=0$, and the cohomology class $[\omega_a^2]= [\omega^2] \in H^{2,2}(\widehat{X},\mathbb{R})$ is independent of $a$.
    \item[(ii)] For each $1 \leq i \leq k$, there is a constant $\lambda_i >0$ and a neighborhood $U_i \supset C_i$, independent of $a$, such that $\omega_{FLY,a} = \lambda_i\omega_{co,a}$ in $U_i$.
    \end{itemize}
    \item  Let $\mathcal{X}\rightarrow \Delta$ be the smoothing, with fibers $X_t$. For $|t|$ sufficiently small there exist hermitian metrics $\omega_{FLY,t}$ on $X_t$ satisfying:
    \begin{itemize}
        \item[(i)] $d\omega_t^2=0$
        \item[(ii)] For each $1 \leq i \leq k$, there is a constant $\lambda_i >0$ and a neighborhood $\mathcal{U}_i \subset \mathcal{X}$, containing the node $\pi(C_i)= p_i \in X_0$ such that, in $\mathcal{U}_i$ we have
        \begin{equation}\label{eq: FLYasympotics}
        \omega_{t}|_{X_t\cap \mathcal{U}_i} = \lambda_i \omega_{co,t} + o(1)
        \end{equation}
        as $t\rightarrow 0$.
    \end{itemize}
\end{enumerate}
\end{thm}

\begin{rk}
    We have state Theorem~\ref{thm: FLY} somewhat informally.  The asymptotics~\eqref{eq: FLYasympotics} should be understood to holds in suitably weighted H\"older spaces as $t\rightarrow 0$. We refer the reader to \cite{FLY} for precise statements.
\end{rk}

By a conformal rescaling, the metrics of Theorem~\ref{thm: FLY} give rise to a solution of~\eqref{eq: introHSbal}.  The proof is by a gluing method, using the K\"ahler Ricci-flat metrics constructed in Section~\ref{sec: metric}. An important observation used in the gluing is that, in a neighborhood of $C_i \subset \widehat{X}$ isomorphic to neighborhood of the zero section in $\mathcal{O}_{\mathbb{P}^1}(-1)^{\oplus 2}$, the K\"ahler metric $\omega_{co,a}$ constructed in Section~\ref{sec: metric} satisfies
\[
\omega_{co,a}^2 = 2\ddbar\left( \phi_a \wedge((4a^2\pi^*\omega_{FS} + \ddbar \phi_a)\right) .
\]
Since this $(2,2)$ form is $\ddbar$-exact, it can be glued to a model metric by introducing a cut-off function.  This observation serves to highlight the added flexibility of working with balanced metrics, rather than K\"ahler metrics.  

One can then ask whether there exist solutions of the Hermitian-Yang-Mills equation~\eqref{eq: introHSYM} with respect to the Fu-Li-Yau metrics.  In this direction Chuan \cite{Chuan} proved

\begin{thm}[Chuan \cite{Chuan}]\label{thm: Chuan}
In the setting of Theorem~\ref{thm: FLY} assume that $E\rightarrow \widehat{X}$ satisfies $c_1(E)=0$ and is slope stable with respect to $\omega$.  Suppose in addition that $E$ is trivial in a neighborhood of the $(-1,-1)$ curves.  Let $\pi:\widehat{X}\rightarrow X_0$ denote the contraction map, and suppose that there exists a family of vector bundles $E_t \rightarrow X_t$, smoothing $\pi_*E$.  Then, for all $|t|$ sufficiently small there exist hermitian metrics $H_t$ on $E_t$ such that
\[
\omega_{FLY,t}^2\wedge F_{H_t}=0
\]
\end{thm}

It seems difficult, in practice, to construct holomorphic vector bundles $E$ satisfying the assumptions of Theorem~\ref{thm: Chuan}, and the cohomological condition~\eqref{eq: chernClassHS}.  The author, with Picard and Yau considered instead the case of $E=T^{1,0}X$, and solved the Hermitian-Yang-Mills equation~\eqref{eq: introHSYM}.

\begin{thm}[C.-Picard-Yau \cite{CPY}]\label{thm: CPY}
In the setting of Theorem~\ref{thm: FLY}, for $|t|\ll1 $ there exists a hermitian metric $H_t$ on $T^{1,0}X_t$ such that
\[
\omega_{FLY,t}^2\wedge F_{H_{t}} = 0.
\]
In particular, $T^{1,0}X_t$ is slope stable with respect to the balanced class $[\omega_{FLY,t}^2] \in H^{2,2}(X_t,\mathbb{R})$.
\end{thm}

Theorem~\ref{thm: CPY} is again a gluing theorem.  The basic observation is that, by the construction of Fu-Li-Yau, the metric $\omega_{FLY,t}$ is close to the Calabi-Yau metric $\omega_{co,t}$.  As emphasized above, the metrics $\omega_{co,t}$ solve the HS system, and hence it is reasonable to try to construct a solution of ~\eqref{eq: introHSYM} which is a small perturbation of the Candelas- de la Ossa metric $\omega_{co,t}$ near the vanishing cycles.  We emphasize that the gluing argument yields quantitative information near the vanishing cycles, and the pair $(\omega_{FLY,t}, H_t)$ approximately solve the anomaly cancellation equation~\eqref{eq: introAnom} near the vanishing cycles; see \cite{CPY}.  We note that when $X_t = \#_{k}(S^3\times S^3)$, Bozkhov \cite{Boz} proved that $T^{1,0}X_t$ is slope stable using purely algebraic methods.  By the work of Li-Yau \cite{LiYau86} this implies the existence of a Hermitian-Yang-Mills connection.

We now discuss the basic strategy of the proof of Theorem~\ref{thm: CPY}, in order to highlight the role of the Candelas-de la Ossa metrics constructed in Section~\ref{sec: metric}.  The proof  proceeds in three steps.
\bigskip

\begin{proof}[Outline of the proof of Theorem~\ref{thm: CPY} ]
\,
\smallskip

\noindent{\bf Step 1}: Let $(\widehat{X},\omega)$ be a compact, K\"ahler Calabi-Yau manifold.  By Yau's theorem, Theorem~\ref{thm: Yau}, and the work of Donaldson \cite{Donaldson}, Uhlenbeck-Yau \cite{UY}, and Li-Yau \cite{LiYau86}, we know that $T^{1,0}\widehat{X}$ is slope semi-stable with respect to $[\omega^2]$.  Hence, we can find hermitian metrics $h_a$ on $T^{1,0}\widehat{X}$ such that
\[
\omega_{FLY,a}^2\wedge F_{h_a}=0
\]
where $F_{h_a}$ denotes the curvature of the Chern connection, and $\omega_{FLY,a}$ denotes the Fu-Li-Yau balanced metric on $\widehat{X}$ constructed by Theorem~\ref{thm: FLY}.  

We now take a limit of the metrics $h_a$ as $a \rightarrow 0$.  Using the ideas of Uhlenbeck-Yau \cite{UY}, one shows that the metrics $h_a$ converge in $C^{\infty}_{loc}(X \setminus \cup_i C_i)$ to hermitian metric on $T^{1,0}X_0\big|_{(X_0)_{reg}}$ which is Hermitian-Yang-Mills with respect to a balanced metric $\omega_0$ defined on $(X_0)_{reg}$.  A key estimate established in this step, using stability and the Uhlenbeck-Yau technique, is that, near the nodal singularities on $X_0$ there is a constant $C$ so that the limit metric $h_0$ satisfies
\[
C^{-1} g_{co,0} \leq h_0 \leq C g_{co,0}.
\]
\bigskip

\noindent{\bf Step 2}:  We analyze the behaviour of the limit $h_0$. 
 By construction, there is a neighborhood $U$ of each ODP singularity $p\in X_0$ such that the balanced metric $\omega_0$ on $U$ is, up to scale, equal to the Calabi-Yau cone metric on the conifold $V_0$.  Since the Calabi-Yau metric on $V_0$ is already Hermitian-Yang-Mills, it is natural to expect a sort of ``infinitesimal uniqueness" statement for $h_0$.  Precisely, we expect that $h_0$ should decay towards a multiple of the Calabi-Yau cone metric near the ODP singularity.  In fact, we need a quantitative version of this statement.

 \begin{thm}\label{thm: tangentCone} \cite{CPY} Let $V_0 = \{ \sum z_i^2 = 0 \} \subseteq \mathbb{C}^4$ and $\omega_{co,0} = i \partial \bar{\partial} r^2$ with $r^3 = \| z \|^2$. Suppose a metric $h_0$ on $T^{1,0}V_0$ solves the equation
   \[
F_{h_0} \wedge \omega_{mod,0}^2 = 0 \quad {\rm on} \quad   V_0\cap \{ 0 < \|z\| <1 \}
   \]
   with bounds $C^{-1} g_{co,0} \leq h_0 \leq C g_{co,0}$. Then
   \[
| h_0 - c_0 g_{co,0}|_{g_{co,0}} \leq C r^\lambda
   \]
   for some constants $c_0>0$, $C>1$, $\lambda \in (0,1)$.
  \end{thm}

This result is established using stability, together with a Poincar\'e type inequality, building on work of Jacob-Walpuski \cite{JacobWal}.
\bigskip

 \noindent{\bf Step 3}: We now glue the metric $h_0$ to the Calabi-Yau metrics $\omega_{co,t}$ on the smoothing of the conifold.  Precisely, as in~\eqref{eq: fiberDiffeo}, we let $F_t$ denote the global extensions of the nearest point projection maps $\Phi_t$ defined in ~\eqref{eq: PhitMap}.  Then $K_{t}:= [(F_t^{-1})^*h_0]^{(1,1)}$ is a hermitian metric defined away from the special Lagrangian vanishing cycles, and which is quantitatively ``approximately Hermitian-Yang-Mills".  In an annulus region around the vanishing cycle, $K_t$ is close to a multiple of $[(\Phi_{t}^{-1})^*g_{co,0}]^{(1,1)}$ by Theorem~\ref{thm: tangentCone}.  On the other hand, by a result of Conlon-Hein \cite{ConlonHein}, $g_{co,t}$ decays at infinity towards $[(\Phi_{t}^{-1})^*g_{co,0}]^{(1,1)}$.  Thus, we can glue a suitably scaled down copy of $g_{co,t}$ to $K_{t}$ to obtain an approximately Hermitian-Yang-Mills metric, since $\omega_{FLY,t}$ is approximately equal to (a multiple of) $g_{co,t}$ near the vanishing cycles.  The result is a hermitian metric on $T^{1,0}X_t$ which is quantitatively close to being Hermitian-Yang-Mills with respect to the balanced metrics $\omega_{FLY,t}$.  A singular perturbation argument implies that we can perturb the metric we have constructed to a genuine Hermitian-Yang-Mills metric.
\end{proof}

Garcia-Fernandez-Molina-Streets \cite{GFMS} have proposed that their string algebroid pluriclosed flow produces birational maps contracting $(-1,-1)$ rational curves as infinite time singularities, at least for appropriate choices of initial data.  Friedman \cite{Fri19}, and Li \cite{Li} have shown that if $\widehat{X} \rightarrow X_0 \leadsto X_t$ is a conifold transition starting from a Calabi-Yau $\widehat{X}$ satisfying the $\ddbar$-lemma, then $X_t$ satisfies the $\ddbar$-lemma. 

\subsection{The ``reverse" conifold transition}

So far we have focused on a conifold transition $\widehat{X}\rightarrow X_0 \leadsto X_t$, in which $\widehat{X}$ is K\"ahler.  Of course, it is equally interesting to consider the ``reverse".  That is, one starts with a family K\"ahler Calabi-Yau manifolds $\mathcal{X} \rightarrow \Delta$, whose central fiber $X_0$ has nodal singularities.  We then get a conifold transition by considering the small resolution $\widehat{X}\rightarrow X_0$.  As discussed in Section~\ref{sec: CY3}, if $\mathcal{X}$ is a generic family of quintic threefolds, then $X_0$ has only a single nodal singularity, which implies by Friedman's theorem that $\widehat{X}$ is non-K\"ahler.  One can then ask whether $\widehat{X}$ admits solutions of the HS system.  Giusti-Spotti \cite{GiustiSpotti2} have used a gluing consutrction, together with the results of Hein-Sun \cite{HeinSun} to produce Chern-Ricci flat balanced metrics on $\widehat{X}$ in this setting

\subsection{Special Lagrangians}

As discussed in Section~\ref{sec: geomConifold}, the vanishing cycles of conifold smoothing $V_0\leadsto V_t$ can be naturally thought of as special Lagrangians; recall Definition~\ref{defn: slag} and Lemma~\ref{lem: sLagMod}.  When $X_0 \leadsto X_t$ is a projective smoothing, then Hein-Sun \cite{HeinSun} showed that the vanishing cycles could be chosen to be special Lagrangian three-spheres with respect to the Calabi-Yau structure.  When $X_0\leadsto X_t$ is a smoothing with non-K\"ahler fibers, then the notion of special Lagrangian still makes sense, but we do not require that the hermitian form $\omega$ is K\"ahler.  In this case, it was shown by Harvey-Lawson \cite{HarveyLawson} that special Lagrangian manifolds minimize conformally rescaled volume functional
\[
L \mapsto \int_{L} |\Omega| dVol_{L}
\]
We remark that these objects were subsequently rediscovered in the physics literature by Becker-Becker-Strominger \cite{BBS}.  It was shown by the author, Picard, Gukov and Yau \cite{CGPY} that for a general non-K\"ahler degeneration, the vanishing cycles can be taken to be special Lagrangian with respect to the Fu-Li-Yau hermitian structure from Theorem~\ref{thm: FLY}.  These extended special Lagrangians are still rather mysterious.  For example, as shown in \cite{CGPY}, their deformation theory differs from the standard deformation theory of special Lagrangians in a K\"ahler manifold \cite{Mclean}.

\end{document}

%% file: derivatives.tex
\newcommand{\nc}{\newcommand}
\nc{\p}{\partial}
\nc{\pb}{\partial_b}
\nc{\pc}{\partial_c}
\nc{\pd}{\partial_d}
\nc{\pe}{\partial_e}
\nc{\pf}{\partial_f}
\nc{\pg}{\partial_g}
\nc{\ph}{\partial_h}
\nc{\pari}{\partial_i}
\nc{\pj}{\partial_j}
\nc{\pk}{\partial_k}
\nc{\pl}{\partial_l}
\nc{\pell}{\partial_\ell}
\nc{\parm}{\partial_m}
\nc{\pn}{\partial_n}
\nc{\po}{\partial_o}
\nc{\pp}{\partial_p}
\nc{\pq}{\partial_q}
\nc{\pr}{\partial_r}
\nc{\ps}{\partial_s}
\nc{\pt}{\partial_t}
\nc{\pu}{\partial_u}
\nc{\pv}{\partial_v}
\nc{\pw}{\partial_w}
\nc{\px}{\partial_x}
\nc{\py}{\partial_y}
\nc{\pz}{\partial_z}

\nc{\pabar}{\partial_{\ol{a}}}
\nc{\pbbar}{\partial_{\ol{b}}}
\nc{\pcbar}{\partial_{\ol{c}}}
\nc{\pdbar}{\partial_{\ol{d}}}
\nc{\pebar}{\partial_{\ol{e}}}
\nc{\pfbar}{\partial_{\ol{f}}}
\nc{\pgbar}{\partial_{\ol{g}}}
\nc{\phbar}{\partial_{\ol{h}}}
\nc{\pibar}{\partial_{\ol{i}}}
\nc{\pjbar}{\partial_{\ol{j}}}
\nc{\pkbar}{\partial_{\ol{k}}}
\nc{\plbar}{\partial_{\ol{l}}}
\nc{\pellbar}{\partial_{\ol{\ell}}}
\nc{\pmbar}{\partial_{\ol{m}}}
\nc{\pnbar}{\partial_{\ol{n}}}
\nc{\pobar}{\partial_{\ol{o}}}
\nc{\ppbar}{\partial_{\ol{p}}}
\nc{\pqbar}{\partial_{\ol{q}}}
\nc{\prbar}{\partial_{\ol{r}}}
\nc{\psbar}{\partial_{\ol{s}}}
\nc{\ptbar}{\partial_{\ol{t}}}
\nc{\pubar}{\partial_{\ol{u}}}
\nc{\pvbar}{\partial_{\ol{v}}}
\nc{\pwbar}{\partial_{\ol{w}}}
\nc{\pxbar}{\partial_{\ol{x}}}
\nc{\pybar}{\partial_{\ol{y}}}
\nc{\pzbar}{\partial_{\ol{z}}}

\nc{\nababar}{\nabla_{\ol{a}}}
\nc{\nabbbar}{\nabla_{\ol{b}}}
\nc{\nabcbar}{\nabla_{\ol{c}}}
\nc{\nabdbar}{\nabla_{\ol{d}}}
\nc{\nabebar}{\nabla_{\ol{e}}}
\nc{\nabfbar}{\nabla_{\ol{f}}}
\nc{\nabgbar}{\nabla_{\ol{g}}}
\nc{\nabhbar}{\nabla_{\ol{h}}}
\nc{\nabibar}{\nabla_{\ol{i}}}
\nc{\nabjbar}{\nabla_{\ol{j}}}
\nc{\nabkbar}{\nabla_{\ol{k}}}
\nc{\nablbar}{\nabla_{\ol{l}}}
\nc{\nabelllbar}{\nabla_{\ol{\ell}}}
\nc{\nabmbar}{\nabla_{\ol{m}}}
\nc{\nabnbar}{\nabla_{\ol{n}}}
\nc{\nabobar}{\nabla_{\ol{o}}}
\nc{\nabpbar}{\nabla_{\ol{p}}}
\nc{\nabqbar}{\nabla_{\ol{q}}}
\nc{\nabrbar}{\nabla_{\ol{r}}}
\nc{\nabsbar}{\nabla_{\ol{s}}}
\nc{\nabtbar}{\nabla_{\ol{t}}}
\nc{\nabubar}{\nabla_{\ol{u}}}
\nc{\nabvbar}{\nabla_{\ol{v}}}
\nc{\nabwbar}{\nabla_{\ol{w}}}
\nc{\nabxbar}{\nabla_{\ol{x}}}
\nc{\nabybar}{\nabla_{\ol{y}}}
\nc{\nabzbar}{\nabla_{\ol{z}}}

\nc{\naba}{\nabla_{a}}
\nc{\nabb}{\nabla_{b}}
\nc{\nabc}{\nabla_{c}}
\nc{\nabd}{\nabla_{d}}
\nc{\nabe}{\nabla_{e}}
\nc{\nabf}{\nabla_{f}}
\nc{\nabg}{\nabla_{g}}
\nc{\nabh}{\nabla_{h}}
\nc{\nabi}{\nabla_{i}}
\nc{\nabj}{\nabla_{j}}
\nc{\nabk}{\nabla_{k}}
\nc{\nabl}{\nabla_{l}}
\nc{\nabm}{\nabla_{m}}
\nc{\nabn}{\nabla_{n}}
\nc{\nabo}{\nabla_{o}}
\nc{\nabp}{\nabla_{p}}
\nc{\nabq}{\nabla_{q}}
\nc{\nabr}{\nabla_{r}}
\nc{\nabs}{\nabla_{s}}
\nc{\nabt}{\nabla_{t}}
\nc{\nabu}{\nabla_{u}}
\nc{\nabv}{\nabla_{v}}
\nc{\nabw}{\nabla_{w}}
\nc{\nabx}{\nabla_{x}}
\nc{\naby}{\nabla_{y}}
\nc{\nabz}{\nabla_{z}}

\nc{\ola}{\ol{a}}
\nc{\olb}{\ol{b}}
\nc{\olc}{\ol{c}}
\nc{\old}{\ol{d}}
\nc{\ole}{\ol{e}}
\nc{\olf}{\ol{f}}
\nc{\olg}{\ol{g}}
\nc{\olh}{\ol{h}}
\nc{\oli}{\ol{i}}
\nc{\olj}{\ol{j}}
\nc{\olk}{\ol{k}}
\nc{\oll}{\ol{l}}
\nc{\olm}{\ol{m}}
\nc{\oln}{\ol{n}}
\nc{\olo}{\ol{o}}
\nc{\olp}{\ol{p}}
\nc{\olq}{\ol{q}}
\nc{\olr}{\ol{r}}
\nc{\ols}{\ol{s}}
\nc{\olt}{\ol{t}}
\nc{\olu}{\ol{u}}
\nc{\olv}{\ol{v}}
\nc{\olw}{\ol{w}}
\nc{\olx}{\ol{x}}
\nc{\oly}{\ol{y}}
\nc{\olz}{\ol{z}}

%% file: benjymacros.tex
\newcommand{\mr}[1]{{\rm #1}}

\usepackage{wrapfig,caption}
\nc{\sn}{\mr{sn}}
\nc{\cn}{\mr{cn}}
\nc{\dn}{\mr{dn}}

\nc{\ol}{\overline}
\nc{\ul}{\underline}
\nc{\iddbar}{\sqrt{-1}\partial \overline{\partial}}

\newcommand{\RNum}[1]{\uppercase\expandafter{\romannumeral #1\relax}}
\nc{\todo}{{{\color{red}\huge {TO DO}}\errormessage} }
\nc{\lob}{\Lambda}
\renewcommand{\Im}{\mr{Im}}

%% file: Lecture_Notes__1_.bbl
\begin{thebibliography}{99}

\bibitem{ABG} L. B. Anderson, C. R. Brodie, and J. Gray, {\em Branes and bundles through conifold transitions and dualities in heterotic string theory}, Phys. Rev. D, 108 (2023), no. 1, Paper No. 

\bibitem{AGS} L. Anderson, J. Gray, and E. Sharpe, {\em Algebroids, heterotic moduli spaces and the Strominger system}, Journal of High Energy Physics no. 7 (2014): 1-40.
    
  \bibitem{AGF} B. Andreas and M. Garcia-Fernandez, {\em Solutions of the Strominger system via stable bundles on Calabi-Yau threefolds}, Communications in Mathematical Physics 315 no.1 (2012): 153-168.
    
  \bibitem{AGF2} B. Andreas and M. Garcia-Fernandez, {\em Heterotic non-K\"ahler geometries via polystable bundles on Calabi-Yau threefolds}, J. Geom. Phys. 62 (2012), no. 2, 183--188.

  \bibitem{ACS} D. Angella, S. Calamai, and C. Spotti, {\em On the Chern-Yamabe problem}, Math. Res. Lett. 24 (2017), no. 3 645--677.

  \bibitem{BBS} K. Becker, M. Becker, A. Strominger, {\em Fivebranes, membranes and non-perturbative string theory}, Nuclear Phys. B, 456 (1–2) (1995), 130–152.

\bibitem{Bog} F. A. Bogomolov,{\em Hamiltonian {K}\"ahlerian manifolds}, Dokl. Akad. Nauk SSSR, 243 (1978), no. 5, 1101--1104.

  \bibitem{Boz} Y.D. Bozhkov, {\em Specific complex geometry of certain complex surfaces and three-folds}, PhD Thesis at the University of Warwick, 1992.
  
  \bibitem{CdlO90} P. Candelas and X. de la Ossa, {\em Comments on conifolds}, Nuclear Phys. B 342 (1990), no. 1, 246-268.

 \bibitem{CdlOHS} P. Candelas, X. de la Ossa, Y.-H. He, and B. Szendr\H oi,  {\em Triadophilia: a special corner in the landscape}, Adv. Theor. Math. Phys., 12 (2008), no. 2, 429--473.

  
  \bibitem{CGH} P. Candelas, P. Green, and T. H\"ubsch, {\em Rolling among Calabi-Yau vacua}, Nuclear Phys. B 330 (1990), no. 1, 49–102.
  
  \bibitem{CHSW} P. Candelas, G. Horowitz, A. Strominger, and E. Witten, {\em Vacuum configurations for superstrings}, Nuclear Phys. B 258 (1985), no. 1, 46–74.

  \bibitem{CGGK} T.-M. Chiang, B. R. Greene, M. and Y. Kanter, {\em Black hole condensation and the web of {C}alabi-{Y}au
              manifolds}, $S$-duality and mirror symmetry (Trieste, 1995), Nuclear Phys. B Proc. Suppl. 46 (1996), 82--95

\bibitem{ChiuSzek} S.-K. Chiu, and G. Sz\'ekelyhidi, {\em Higher regularity for singular {K}\"ahler-{E}instein metrics}, Duke Math. J. 172 (2023), no. 18, 3521--3558
  
  \bibitem{Chuan} M.-T. Chuan, {\em Existence of Hermitian-Yang-Mills metrics under conifold transitions}, Comm. Anal. Geom. 20 (2012), no. 4, 677–749. 

  \bibitem{Clemens} C. Clemens, {\em Double solids}, Adv. in Math. 47, no. 2 (1983), 107–230.
  
  \bibitem{CPY} T.C. Collins, S. Picard, and S. -T. Yau, {\em Stability of the tangent bundle through conifold transitions}, Comm. Pure Appl. Math. 77 (2024), 284-371.

  \bibitem{CPY2} T.C. Collins, S. Picard, and S.-T. Yau, {\em The {S}trominger system in the square of a {K}\"ahler class}, Pure Appl. Math. Q., 21 (2025), no. 3, 1015--1035

  \bibitem{CGPY} T.C. Collins, S. Gukov, S. Picard, and S.-T. Yau, {\em Special {L}agrangian cycles and {C}alabi-{Y}au transitions}, Comm. Math. Phys., 401 (2023), no. 1, 769--802.
  
 \bibitem{ConlonHein} R. Conlon and H.J. Hein, {\em Asymptotically conical Calabi-Yau manifolds, I}, Duke Math. J. 162 (2013), 2855-2902.
 
\bibitem{Donaldson} S.K. Donaldson, {\em Anti self-dual Yang-Mills connections over complex algebraic surfaces and stable vector bundles}, Proc. London Math. Soc. (3) 50 (1985), no.1, 1-26

\bibitem{Fei} T. Fei, {\em A construction of non-Kahler Calabi-Yau manifolds and new solutions to the Strominger system}, Adv. Math. 302 (2016), 529–550.

\bibitem{Fei2} T. Fei, {\em Some Torsional Local Models of Heterotic Strings}, Comm. Anal. Geom. 25 (2017), no. 5, 941--968

\bibitem{FeiPhong} T. Fei and D.H. Phong, {\em Unification of the K\"ahler-Ricci and Anomaly flows}, In Surveys in differential geometry 2018. Differential geometry, Calabi-Yau theory, and general relativity, 89–103, Surv. Differ. Geom., 23, Int. Press, Boston. (2020)

\bibitem{FeiPic} T. Fei and S. Picard, {\em Anomaly Flow and T-duality}, Pure and Applied Mathematics Quarterly, Vol. 17, No. 3 (2021), 1083-1112.

\bibitem{FeiYau} T. Fei and S.-T. Yau, {\em Invariant solutions to the Strominger system on complex Lie groups and their quotients}, Comm. Math. Phys. 338 (2015), no. 3, 1183–1195.

\bibitem{FHP} T. Fei, Z. Huang, S. Picard, {\em A construction of infinitely many solutions to the Strominger system}, Journal of Differential Geometry 117(1), 23–39.
  
\bibitem{FPPZ} T. Fei, D.H. Phong, S. Picard and X.-W. Zhang, {\em Estimates for a geometric flow for the Type IIB string}, Mathematische Annalen Vol 382 (2022), 1935–1955.

\bibitem{FIUV} M. Fernandez, S. Ivanov, L. Ugarte, and R. Villacampa, {\em Non-Kaehler heterotic string compactifications with non-zero fluxes and constant dilaton}, Comm. Math. Phys. 288 (2009), no. 2, 677-697.

\bibitem{FIUV2} M.Fernandez, S. Ivanov, L. Ugarte, and R. Villacampa, {\em Non-Kahler heterotic string solutions with non-zero fluxes and non-constant dilaton}, Journal of High Energy Physics 06, (2014):73.

\bibitem{FGV} A. Fino, G. Grantcharov and L. Vezzoni, {\em Solutions to the Hull–Strominger System with Torus Symmetry}. Commun. Math. Phys. 388 (2021), 947–967.

\bibitem{Fri86} R. Friedman, {\em Simultaneous resolution of threefold double points}, Math. Ann. 274 (1986), no. 4, 671–689.

\bibitem{Fri19} R. Friedman, {\em The {$\partial\overline\partial$}-lemma for general {C}lemens
              manifolds}, Pure Appl. Math. Q. 15 (2019), no. 4, 1001--1028

\bibitem{Friedman} R. Friedman, {\em On threefolds with trivial canonical bundle}, Complex geometry and Lie theory (Sundance, UT, 1989), 103–134. Proc. Sympos. Pure Math., 53, Amer. Math. Soc., Providence, RI, 1991.

\bibitem{Friedman2} R. Friedman, {\em Unobstructed deformations for singular Calabi-Yau varieties}, preprint, arXiv:2506.09857

\bibitem{FrLa1} R. Friedman, and R. Laza, {\em Deformations of singular {F}ano and {C}alabi-{Y}au varieties}, J. Differential Geom. 131 (2025), no. 1, 65--131

\bibitem{FrLa2} R. Friedman, and R. Laza, {\em Deformations of {C}alabi-{Y}au varieties with isolated log
              canonical singularities}, Int. Math. Res. Not. IMRN (2025), no. 10.

\bibitem{FrLa3} R. Friedman, and R. Laza, {\em Deformations of {C}alabi-{Y}au varieties with {$k$}-liminal
              singularities}, Forum Math. Sigma (2024), vol. 12

\bibitem{FrLa4} R. Friedman, and R. Laza, {\em Higher {D}u {B}ois and higher rational singularities}, Appendix by Morihiko Saito, Duke Math. J. 173 (2024), no. 10, 1839--1881

\bibitem{FrLa5} R. Friedman, and R. Laza, {\em The higher {D}u {B}ois and higher rational properties for
              isolated singularities}, J. Algebraic Geom. 33 (2024), no. 3, 493--520

\bibitem{FY} J. Fu, and S.-T Yau, {\em The theory of superstring with flux on non-Kahler manifolds and the complex Monge-Amp\`ere equation} J. Differential Geom. 78 (2008), no. 3, 369– 428.

  \bibitem{FLY}  J. Fu, J. Li, and S.-T. Yau, {\em Balanced metrics on non-Kahler Calabi-Yau threefolds}, J. Differential Geom. 90 (2012), no. 2, 81-129.

  \bibitem{FTY} J. Fu, L.-S. Tseng, and S.-T. Yau, {\em Local heterotic torsional models}, Comm. Math. Phys. 289 (2009), no. 3, 1151–1169.  

  \bibitem{XinFu}, X. Fu,{\em Uniqueness of tangent cone of {K}\"ahler-{E}instein metrics on
              singular varieties with crepant singularities}, Math. Ann. 388 (2024), no. 3, 3229--3258

\bibitem{GF-survey} M. Garcia-Fernandez, {\em Lectures on the Strominger system}, Travaux math\'ematiques. Vol XXIV, 7--61, Trav. Math., 24, Fac. Sci. Technol. Commun. Univ. Luxemb., Luxembourg, 2016

\bibitem{GM} M. Garcia-Fernandez, and R. Gonzalez Molina, {\em Futaki invariants and Yau's conjecture on the Hull-Strominger system}, preprint.

\bibitem{GRT} M. Garcia-Fernandez, R. Rubio, C. Tipler {\em Gauge theory for string algebroids}, to appear in J. Diff. Geometry, arXiv:2004.11399.
  
\bibitem{GRST} M. Garcia-Fernandez, R. Rubio, C. Shahbazi, C. Tipler {\em Canonical metrics on holomorphic Courant algebroids}, to appear in Proc. London Math. Soc. arXiv:1803.01873.

\bibitem{GFMS} M. Garcia-Fernandez, J. Streets, and R. Molina {\em Pluriclosed flow and the Hull-Strominger system}, preprint,  arXiv:2408.11674 


\bibitem{GiustiSpotti} F. Giusti, and C. Spotti, {\em A K\"ummer construction for Chern-Ricci flat balanced manifolds}, preprint, arXiv:2309.12909

\bibitem{GiustiSpotti2} F. Giusti, and C. Spotti, {\em Chern-Ricci flat balances metrics on small resolutions of Calabi-Yau threefolds}, preprint, arXiv:2301.11636


 \bibitem{GMPT} M. Gra\~na, R. Minasian, M. Petrini, and A. Tomasiello, {\em Generalized structures of {$\mathcal{N}=1$} vacua}, J. High Energy Phys. (2005), no. 11
  
  \bibitem{GH1} P.S. Green and T. H\"ubsch, {\em Connecting Moduli Spaces of Calabi-Yau Threefolds}, Commun. Math. Phys. 119 (1988), 431-441.
  
\bibitem{GH2} P.S. Green and T. H\"ubsch, {\em Possible Phase Transitions among Calabi-Yau Compactifications}, Phys. Rev. Lett. 61 (1988), 1163.

\bibitem{GMS} B.R. Greene, D.R. Morrison, and A. Strominger, {\em Black hole condensation and the unification of string vacua}, Nuclear Phys. B 451 (1995), no. 1-2, 109–120.

\bibitem{Gross1} M. Gross, {\em Deforming {C}alabi-{Y}au threefolds}, Math. Ann. 308 (1997), no. 2, 187--220.

\bibitem{Gross2} M. Gross, {\em Primitive {C}alabi-{Y}au threefolds}, J. Differential Geom. 45 (1997), no. 2, 288--318

\bibitem{HarveyLawson} R. Harvey and B. Lawson, {\em Calibrated geometries}, Acta Math. 148 (1982), 47--157.  

\bibitem{HeinSun} H.-J. Hein, and S. Sun, {\em Calabi-{Y}au manifolds with isolated conical singularities}, Publ. Math. Inst. Hautes \'Etudes Sci.,126 (2017), 73--130


\bibitem{ClayMS} K. Hori, S. Katz, A. Klemm, R. Pandharipande, R. Thomas, C. Vafa, R. Vakil, and E. Zaslow, {\em Mirror symmetry}, Clay Mathematics Monographs, Vol. 1, American Mathematical Society, Providence, RI; Clay Mathematics Institute, Cambridge, MA, 2003.

\bibitem{Hull}  C.M. Hull, {\em Compactifications of the heterotic superstring}, Phys. Lett. B 1978 (1986), no. 4, 357--364.

\bibitem{HuybrechtsK3} D. Huybrechts, {\em Lectures on {K}3 surfaces}, Cambridge Studies in Advanced Mathematics, 158, Cambridge University Press, Cambridge, 2016.

\bibitem{Imagi} Y. Imagi, {\em Deformations of compact Calabi-Yau conifolds}, arXiv:2504.03243

\bibitem{Ivanov} S. Ivanov {\em Heterotic supersymmetry, anomaly cancellation and equations of motion}, Phys. Lett. B  685 (2010), no. 2--3, 190--196.

\bibitem{JacobWal} A. Jacob, T. Walpuski, {\em Hermitian-Yang-Mills metrics on reflexive sheaves over asymptotically cylindrical Kahler manifolds}, Comm. Partial Differential Equations 43 (2018), no. 11, 1566--1598. 

\bibitem{Song} J. Song, {\em On a conjecture of {C}andelas and de la {O}ssa}, Comm. Math. Phys. 334 (2015), no. 2, 697--717

\bibitem{Kawamata} Y. Kawamata, {\em Unobstructed deformations– a remark on a paper of Z. Ran}, J. Alg. Geom. (1992) no. 1, 183--190.

\bibitem{Kont} M. Kontsevich, {\em Mirror symmetry in dimension {$3$}}, S\'eminaire Bourbaki, Vol.\ 1994/95, Ast\'erisque, 237 (1996), Exp. No. 801, 5, 275--293

\bibitem{Li}C. Li {\em Polarized {H}odge structures for {C}lemens manifolds}, Math. Ann. 389 (2024), no. 1, 525--541

\bibitem{LiYau86} J. Li and S.T. Yau, {\em Hermitian-Yang-Mills connections on non-Kahler manifolds}, Mathematical aspects of string theory, Adv. Ser. Math. Phys., World Sci. Publishing (1986), 560-573.

\bibitem{LY05} J. Li and S.T. Yau, {\em The existence of supersymmetric string theory with torsion}, J. Diff. Geom. 70 no. 1 (2005), 143-181.

\bibitem{LuTian} P. Lu and G. Tian, {\em The complex structure on a connected sum of $S^3 \times S^3$ with trivial canonical bundle}, Mathematische Annalen 298 (1994), 761–764.

\bibitem{LuTian2} P. Lu and G. Tian, {\em Complex structures on connected sums of S3 × S3}, Manifolds and geometry, 284-293, Sympos. Math., XXXVI, Cambridge Univ. Press, Cambridge, 1996.

\bibitem{Morrison2} D.R. Morrison, {\em Mirror symmetry and rational curves on quintic threefolds: a
              guide for mathematicians}, J. Amer. Math. Soc. 6 (1993), no. 1, 223--247


\bibitem{Mclean} R.C. McLean, {\em Deformations of calibrated submanifolds}, Comm. Anal. Geom 6 (1998), 705–-747.

\bibitem{Michelsohn} M.L. Michelsohn, {\em On the existence of special metrics in complex geometry}, Acta Math. 149 (1982), 261-295.

\bibitem{Morrison} D.R. Morrison, {\em Through the looking glass}, Mirror symmetry, {III} ({M}ontreal, {PQ}, 1995),AMS/IP Stud. Adv. Math. Vol. 10, 263--277, Amer. Math. Soc., Providence, RI, 1999.

\bibitem{NamSt} Y. Namikawa, and J. Steenbrink {\em Global smoothing of Calabi–Yau 3–fold} Invent. Math. 122 (1995), 403--419.

\bibitem{OUV} A. Otal, L. Ugarte and R. Villacampa, {\em Invariant solutions to the Strominger system and the heterotic equations of motion}, Nuclear Phys. B, Vol. 920, p. 442-474 (2017). 

\bibitem{PhongSurvey} D.H. Phong, {\em Geometric flows from unified string theories}, Surveys in differential geometry 2022. {E}ssays on geometric flows---celebrating 40 years of {R}icci flow, Surv. Differ. Geom. 27, 75--102, Int. Press, Somerville, MA, 2024
  

\bibitem{PPZ18b} D.H. Phong, S. Picard, X.-W. Zhang, {\it Anomaly flows}, Comm. Anal. Geom. 26 (2018), no. 4, 955-1008.  
  
\bibitem{PPZ18a} D.H. Phong, S. Picard, X.-W. Zhang, {\it Geometric flows and Strominger systems}, Math. Z. 288 (2018), 101-113.

\bibitem{PPZ18} D. H. Phong, S. Picard, and X. Zhang, {\em The Anomaly flow and the Fu-Yau equation}, Ann. PDE 4 (2018), no. 2, Paper No. 13, 60 pp.

\bibitem{PPZ19} D. H. Phong, S. Picard, and X. Zhang, {\em A flow of conformally balanced metrics with Kahler fixed points}, Math. Ann. 374 (2019), no. 3-4, 2005-2040.


\bibitem{PicardSurvey1} S. Picard,{\em Calabi-{Y}au manifolds with torsion and geometric flows}, Complex non-{K}\"ahler geometry, Lecture Notes in Math., vol. 2246, 57--120, Springer, Cham, 2019 

\bibitem{PicardLNS} S. Picard, {\em Calabi-Yau threefolds across quadratic singularities}, preprint, arXiv:2501.19313

\bibitem{PicardSurvey2} S. Picard, {\em The Strominger system and flows by the Ricci tensor}, Surveys in Differential Geometry 27 (2022), no. 1, 103-145.

\bibitem{PicardMcOrist} S. Picard and J. McOrist, {\em Stringy Corrections to Heterotic SU(3)-Geometry}, preprint, arXiv:2507.02388

\bibitem{Ran} Z. Ran, {\em Deformations of Calabi-Yau Kleinfolds}, Essays on mirror manifolds, Int. Press, Hong Kong, 1992, pp. 451–-457

\bibitem{Reid} M. Reid, {\em The moduli space of 3-folds with $K=0$ may nevertheless be irreducible}, Math. Ann. 278 (1987), no. 1-4, 329–-334.

\bibitem{RolTh} S. Rollenske, and R. Thomas, {\em Smoothing nodal {C}alabi-{Y}au {$n$}-folds}, J. Topol. 2 (2009), no. 2, 405--421.

\bibitem{Rossi} M. Rossi, {\em Geometric Transitions}, Journal of Geometry and Physics Volume 56, Issue 9 (2006), 1940–1983.

\bibitem{Schoen} C. Schoen, {\em On the geometry of a special determinantal hypersurface associated to the {M}umford-{H}orrocks vector bundle}, J. Reine Angew. Math., 364 (1986), 85--111.

\bibitem{STY} I. Smith, R.P. Thomas, S.-T. Yau, {\em Symplectic Conifold Transitions}, J. Differential Geom. 62(2) (2002), 209-242.


  \bibitem{Stromingerholes} A. Strominger, {\em Massless black holes and conifolds in string theory}, Nuclear Phys. B 451 (1995), no. 1-2, 96–108.

\bibitem{Strominger} A. Strominger, {\em Superstrings with torsion}, Nuclear Phys. B 274 (1986), no. 2, 253–284.

  \bibitem{SYZ} A. Strominger, S.-T. Yau, E. Zaslow, {\em Mirror symmetry is T-duality}, Nuclear Phys. B 479 (1996), 243–259.

  \bibitem{STW} G. Szekelyhidi, V. Tosatti, and B. Weinkove, {\em Gauduchon metrics with prescribed volume form}, Acta Math. 219 (2017), no. 1, 181--211.
  
  \bibitem{Tian} G. Tian, {\em Smoothing 3-folds with trivial canonical bundle and ordinary double points}, Essays on mirror manifolds, 458–479, Int. Press, Hong Kong, 1992. 

  \bibitem{Tian2} G. Tian, {\em Smoothness of the universal deformation space of compact {C}alabi-{Y}au manifolds and its {P}etersson-{W}eil metric}, Mathematical aspects of string theory ({S}an {D}iego,
  {C}alif., 1986), Adv. Ser. Math. Phys., Vol. 1, 629--646, World Sci. Publishing, Singapore, 1987.
  

 \bibitem{TianYau2} G. Tian, G. and S.-T. Yau,{\em Three-dimensional algebraic manifolds with {$C_1=0$} and
              {$\chi=-6$}}, Mathematical aspects of string theory ({S}an {D}iego, {C}alif., 1986), Adv. Ser. Math. Phys., vol. 1, 543--559, World Sci. Publishing, Singapore, 1987

  \bibitem{Todorov} A. N. Todorov, {\em The {W}eil-{P}etersson geometry of the moduli space of {${\rm SU}(n\geq 3)$} ({C}alabi-{Y}au) manifolds. {I}}, Comm. Math. Phys., 126 (1989), no. 2, 325--346

  \bibitem{Tomasiello} A. Tomasiello, {\em Reformulating supersymmetry with a generalized {D}olbeault
              operator}, J. High Energy Phys.(2008), no. 2

  \bibitem{Tosatti} V. Tosatti, {\em Non-{K}\"ahler {C}alabi-{Y}au manifolds}, Analysis, complex geometry, and mathematical physics: in honor of {D}uong {H}. {P}hong, Contemp. Math. 644, 261--277, Amer. Math. Soc., Providence, RI, 2015

  \bibitem{TW} V. Tosatti and B. Weinkove,{\em  The complex Monge-Amp\`ere equation on compact Hermitian manifolds}, J. Amer. Math. Soc. 23 (2010), no. 4, 1187--1195.

  \bibitem{TW2} V. Tosatti and B. Weinkove, {\em The Monge-Amp\`ere equation for $(n-1)$-plurisubharmonic functions on a compact K\"ahler manifold}, J. Amer. Math. Soc. 30 (2017), no.2, 311--346.

\bibitem{TsengYau} L.-S. Tseng and S.-T. Yau, {\em Non-Kaehler Calabi-Yau manifolds}, in Strings Math 2011, 241-254, Proceedings of Symposia in Pure Mathematics, 85, Amer. Math. Soc., Providence, RI (2012).
  
 \bibitem{UY}  K. Uhlenbeck and S.T. Yau, {\em On the existence of Hermitian-Yang-Mills connections in stable vector bundles}, Comm. Pure Appl. Math. 39-S (1986), 257–293; 42 (1989), 703–707.

 \bibitem{Wall} C. T. C Wall, {\em Classification problems in differential topology. {V}. {O}n
              certain {$6$}-manifolds}, Invent. Math. 1 (1966), 355--374; corrigendum, ibid. 2 (1966), 306

 \bibitem{Wang} S.-S. Wang, {\em On the connectedness of the standard web of {C}alabi-{Y}au
              3-folds and small transitions}, Asian J. Math., 22 (2018), no. 6, 981--1003
    
\bibitem{Yau78} S.-T. Yau, {\em On the Ricci curvature of a compact K\"ahler manifold and the complex Monge-Amp\`ere equation.} I, Comm. Pure Appl. Math. 31 (1978) 339-411.
\end{thebibliography}
